\documentclass[10pt,reqno]{amsart}%
\usepackage[english, activeacute]{babel}
\usepackage{epsfig}
\usepackage{amssymb}
\usepackage{latexsym}
\usepackage{amsmath}
\usepackage{amsthm}
\usepackage{amsfonts}
\usepackage{graphicx}%
\makeatletter \numberwithin{equation}{section}
\newcommand{\equ}[1]{(\ref{#1})}

\providecommand{\abs}[1]{\lvert#1 \rvert}
\providecommand{\norm}[1]{\lVert#1 \rVert}

\newcommand{\ve}{\varepsilon}
\theoremstyle{plain}
\newtheorem{claim}{Claim}

\newtheorem{teo}{Theorem}
\newtheorem{prop}{Proposition}[section]

\newtheorem{lema}{Lemma}[section]
\newtheorem{remark}{Remark}[section]

\theoremstyle{definition}

\begin{document}
\date{}
\title[Singular limits for Bi-Laplacian operator in $\mathbb{R}^{4}$ ]{Singular limits for the Bi-Laplacian operator with exponential nonlinearity in
$\mathbb{R}^{4}$}
\author{M\'onica Clapp}
\address{Instituto de Matem\'{a}ticas, Universidad Nacional Aut\'{o}noma de M\'{e}xico,
Circuito Exterior, C.U., 04510 M\'{e}xico DF, M\'{e}xico}
\email{mclapp@matem.unam.mx}
\author{Claudio Mu\~noz}
\address{Departamento de Ingenier\'ia Matem\'atica, Universidad de Chile, Casilla 170,
Correo 3, Santiago, Chile.}
\email{cmunoz@dim.uchile.cl }
\author{Monica Musso}
\address{Departamento de Matematica, Pontificia Universidad Catolica de
Chile, Avenida Vicuna Mackenna 4860, Macul, Santiago, Chile and
Dipartimento di Matematica, Politecnico di Torino, Corso Duca degli
Abruzzi, 24 -- 10129 Torino, Italy.} \email{mmusso@mat.puc.cl}

\begin{abstract}
Let $\Omega$ be a bounded smooth domain in $\mathbb{R}^{4}$ such that for some
integer $d\geq1$ its $d$-th singular cohomology group with coefficients in
some field is not zero, then problem
\[
\quad%
\begin{cases}
\Delta^{2}u-\rho^{4}k(x)e^{u}=0 & \hbox{ in }\Omega,\\
u=\Delta u=0 & \hbox{ on }\partial\Omega,
\end{cases}
\]
has a solution blowing-up, as $\rho\rightarrow0$, at $m$ points of $\Omega$,
for any given number $m$.

\end{abstract}
\maketitle




\section{Introduction and statement of main results}

\noindent Let $\Omega$ be a bounded and smooth domain in $\mathbb{R}^{4}$. We
are interested in studying existence and qualitative properties of positive
solutions to the following boundary value problem
\begin{equation}
\quad%
\begin{cases}
\Delta^{2}u-\rho^{4}k(x)e^{u}=0 & \hbox{ in }\Omega,\\
u=\Delta u=0 & \hbox{ on }\partial\Omega,
\end{cases}
\label{In1}%
\end{equation}
where $k\in C^{2}(\bar{\Omega})$ is a non-negative, not identically zero
function, and $\rho>0$ is a small, positive parameter which tends to $0$.

In a four-dimensional manifold, this type of equations and similar ones arise
from the problem of prescribing the so-called $Q$-curvature, which was
introduced in \cite{6}. More precisely, given $(M,g)$ a four-dimensional
Riemannian manifold, the problem consists in finding a conformal metric
$\tilde{g}$ for which the corresponding $Q$-curvature $Q_{\tilde{g}}$ is
a-priori prescribed. The $Q$-curvature for the metric $g$ is defined as
\[
Q_{g}=-{\frac{1}{2}}\left(  \Delta_{g}R_{g}-R_{g}^{2}+3|{\mbox {Ric}}_{g}%
|^{2}\right)  ,
\]
where $R_{g}$ is the scalar curvature and ${\mbox {Ric}}_{g}$ is the Ricci
tensor of $(M,g)$. Writing $\tilde{g}=e^{2w}g$, the problem reduces to finding
a scalar function $w$ which satisfies
\begin{equation}
P_{g}w+2Q_{g}=2Q_{\tilde{g}}e^{4w}, \label{pane}%
\end{equation}
where $P_{g}$ is the Paneitz operator \cite{38,cy1} defined as
\[
P_{g}w=\Delta_{g}^{2} w+div\left(  {\frac{2}{3}}R_{g}g-2{\mbox {Ric}}%
_{g}\right)  dw.
\]
Problem (\ref{pane}) is thus an elliptic fourth-order partial differential
equation with exponential non-linearity. Several results are already known for
this problem \cite{c1,cy1} and related ones \cite{ARS,DR,MS}. When the metric
$g$ is not Riemannian, the problem has been recently treated by Djadli and
Malchiodi in \cite{DM} via variational methods.

In the special case where the manifold is the Euclidean space and $g$ is the
Euclidean metric, we recover the equation in (\ref{In1}), since (\ref{pane})
takes the simplified form
\[
\Delta^{2}w-2Qe^{4w}=0.
\]

Problem (\ref{In1}) has a variational structure. Indeed, solutions of
(\ref{In1}) correspond to critical points in $H^{2}(\Omega)\cap H_{0}%
^{1}(\Omega)$ of the following energy functional
\[
J_{\rho}(u)={\frac{1}{2}}\int_{\Omega}|\Delta u|^{2}-\rho^{4}\int_{\Omega
}k(x)e^{u}.
\]
For any $\rho$ sufficiently small, the functional above has a local minimum
which represents a solution to (\ref{In1}) close to $0$. Furthermore, the
Moser-Trudinger inequality assures the existence of a second solution, which
can be obtained as a mountain pass critical point for $J_{\rho}$. Thus, as
$\rho\rightarrow0$, this second solution turns out not to be bounded. The aim
of the present paper is to study multiplicity of solutions to (\ref{In1}), for
$\rho$ positive and small, under some topological assumption on $\Omega,$ and
to describe the asymptotic behaviour of such solutions as the parameter $\rho$
tends to zero. Indeed, we prove that, if some cohomology group of $\Omega$ is
not zero, then given any integer $m$ we can construct solutions to (\ref{In1})
which concentrate and blow-up, as $\rho\rightarrow0$, around some given $m$
points of the domain. These are the singular limits.

Let us mention that concentration phenomena of this type, in domains with
topology, appear also in other problems. As a first example, the
two-dimensional version of problem (\ref{In1}) is the boundary value problem
associated to Liouville%
\'{}%
s equation \cite{l}
\begin{equation}
\left\{
\begin{array}
[c]{l}%
\Delta u+\rho^{2}\,k(x)\,e^{u}=0,\quad\mbox{in}\;\Omega,\cr\cr u=0,\quad
\mbox{on}\;\partial\Omega,
\end{array}
\right.  \label{l}%
\end{equation}
where $k(x)$ is a non-negative function and now $\Omega$ is a smooth bounded
domain in $\mathbb{R}^{2}$. In \cite{KMP} it is proved that problem (\ref{l})
admits solutions concentrating, as $\rho\rightarrow0$, around some given set
of $m$ points of $\Omega$, for any given integer $m$, provided that $\Omega$
is not simply connected. See also \cite{bp,egp,e,cl2,bm,wei,ns,s,w,t,tt} for
related results. A similar result holds true for another semilinear elliptic
problem, still in dimension $2$, namely
\begin{equation}
\left\{
\begin{array}
[c]{l}%
\Delta u+u^{p}=0,\quad u>0,\quad\mbox{in}\;\Omega,\cr\cr u=0,\quad
\mbox{on}\;\partial\Omega,
\end{array}
\right.  \label{l1}%
\end{equation}
where $p$ now is a parameter converging to $+\infty$. Again in this situation,
if $\Omega$ is not simply connected, then for $p$ large there exists a
solution to (\ref{l1}) concentrating around some set of $m$ points of $\Omega
$, for any positive integer $m$ \cite{EMP}.

In higher dimensions, the analogy is with the classical Bahri-Coron problem.
In \cite{bc}, Bahri and Coron show that, if $N\geq3$ and $\Omega
\subset\mathbb{R}^{N}$ is a bounded domain, then the presence of topology in
the domain guarantees existence of solutions to
\begin{equation}
\left\{
\begin{array}
[c]{l}%
\Delta u+u^{{\frac{N+2}{N-2}}}=0,\quad u>0,
\quad\mbox{in}\;\Omega,\cr\cr  u=0,\quad\mbox{on}\;\partial\Omega.
\end{array}
\right.
\end{equation}
Partial results in this direction are also known in the slightly
super critical version of Bahri-Coron's problem, namely
\begin{equation}
\left\{
\begin{array}
[c]{l}%
\Delta u+u^{{\frac{N+2}{N-2}}+\varepsilon}=0,\quad\mbox{in}\;\Omega
,\cr\cr u>0\;,\quad u=0,\quad\mbox{on}\;\partial\Omega,
\end{array}
\right.  \label{bn}%
\end{equation}
with $\varepsilon>0$ small. In \cite{dfm} it is proved that, under the
assumption that $\Omega$ is a bounded smooth domain in $\mathbb{R}^{N}$ with a
sufficiently small hole, then a solution to (\ref{bn}) exhibiting
concentration in two points is present. See also \cite{blr,r,ddm,pr}.

The main point of this paper is to show that the presence of topology in the
domain implies strongly existence of blowing-up solutions for problem
(\ref{In1}).

Let $H^{\ast}:=H^{\ast}($ $\cdot$ $;\mathbb{K})$ denote singular cohomology
with coefficients in a field $\mathbb{K}$. We also denote by $H^{d}(\Omega)$
the $d$-th cohomology group in the field $\mathbb{K}$. We shall prove the
following \medskip

\begin{teo}
\label{teo1} Assume that there exists $d\geq1$ such that $H^{d}(\Omega
)\not =0$ and that $\inf_{\Omega}k>0$. Then, given any integer $m\geq1$, there
exists a family of solutions $u_{\rho}$ to Problem \eqref{In1}, for $\rho$
small enough, with the property that
\[
\lim_{\rho\rightarrow0}\rho^{4}\int_{\Omega}k(x)e^{u_{\rho}(x)}\,dx=64\,\pi
^{2}\,m.
\]
Furthermore, there are $m$ points $\xi_{1}^{\rho},\ldots,\xi_{m}^{\rho}$ in
$\Omega$, separated at uniform positive distance from each other and from the
boundary as $\rho\rightarrow0,$ for which $u_{\rho}$ remains uniformly bounded
on $\Omega\setminus\cup_{j=1}^{m}B_{\delta}(\xi_{j}^{\rho})$ and
\[
\sup_{B_{\delta}(\xi_{j}^{\rho})}u_{\rho}\rightarrow+\infty,
\]
for any $\delta>0$.
\end{teo}

As a simple example, we can say that any bounded domain in $\mathbb{R}^{4}$
that is not simply connected satisfies $H^{1}(\Omega)\neq0$ and thus above
theorem ensures existence of multiple solutions for Problem (\ref{In1}) for
$\rho$ small enough.

The general behaviour of arbitrary families of solutions to
\equ{In1} has been studied by C.S. Lin and J.-C. Wei in
\cite{referee}, where they show that, when blow-up occurs for
\equ{In1} as $\rho \to 0$, then it is located at a finite number of
peaks, each peak being isolated and carrying the energy $64 \pi^2$
(at a peak, $u \to +\infty$ and outside a peak, $u$ is bounded). See
\cite{linwei1} and \cite{linwei2} for related results.

We shall see that the sets of points where the solution found in
Theorem \ref{teo1} blows-up can be characterized in terms of Green's
function for the biharmonic operator in $\Omega$ with the
appropriate boundary conditions. Let $G(x,\xi)$ be the Green
function defined by
\begin{equation}
\quad%
\begin{cases}
\Delta_{x}^{2}G(x,\xi)=64\pi^{2}\delta_{\xi}(x), & x\in\Omega,\\
\quad G(x,\xi)=\Delta_{x}G(x,\xi)=0, & x\in\partial\Omega
\end{cases}
\label{Un1}%
\end{equation}
and let $H(x,\xi)$ be its \emph{regular part}, namely, the smooth function
defined as
\[
H(x,\xi):=G(x,\xi)+8\log\,\abs{x-\xi}.
\]

\noindent The location of the points of concentration is related to the set of
critical points of the function
\begin{equation}
\varphi_{m}(\xi)=-\sum_{j=1}^{m}\left\{  2\log k(\xi_{j})+H(\xi_{j},\xi
_{j})\right\}  -\sum_{i\neq j}G(\xi_{i},\xi_{j}), \label{Rob}%
\end{equation}
defined for points $\xi=(\xi_{1},\dots,\xi_{m})$ such that $\xi_{i}\in\Omega$
and $\xi_{i}\neq\xi_{j}$ if $i\neq j$.

In \cite{BaPa} the authors prove that for each \textit{nondegenerate} critical
point of $\varphi_{m}$ there exists a solution to (\ref{In1}), for any small
$\rho$, which concentrates exactly around such critical point as
$\rho\rightarrow0$. We shall show the existence of a solution under a weaker
assumption, namely, that $\varphi_{m}$ has a \textit{minimax value in an
appropriate subset.}

More precisely, we consider the following situation. Let $\Omega^{m}$ denote
the cartesian product of $m$ copies of $\Omega.$ Note that in any compact
subset of $\Omega^{m}$, we may define, without ambiguity,
\[
\varphi_{m}(\xi_{1},\ldots,\xi_{m})=-\infty\quad\hbox{if }\xi_{i}=\xi
_{j}\hbox{ for some }i\neq j.
\]
We shall assume that there exists an open subset $U$ of $\Omega$ with smooth
boundary, compactly contained in $\Omega,$ and such that $\inf_{U}k>0,$ with
the following properties:

\noindent\textbf{P1)} $U^{m}$\emph{ contains two closed subsets }$B_{0}\subset
B$\emph{ such that}
\[
\sup_{\xi\in B_{0}}\varphi_{m}(\xi)<\inf_{\gamma\in\Gamma}\sup_{\xi\in
B}\varphi_{m}(\gamma(\xi))=:c_{0},
\]
\emph{where} $\Gamma:=\{\gamma\in\mathcal{C}(B,\overline{U}^{m}):\gamma
(\xi)=\xi$ for every $\xi\in B_{0}\}.$

\noindent\textbf{P2)} \emph{For every} $\xi=(\xi_{1},...,\xi_{m})\in\partial
U^{m}$ \emph{with} $\varphi_{m}(\xi)=c_{0}$, \emph{there exists an}
$i\in\{1,...,m\}$ \emph{such that}%
\[%
\begin{array}
[c]{ll}%
\nabla_{\xi_{i}}\varphi_{m}(\xi)\not =0 & \text{if }\xi_{i}\in U,\\
\nabla_{\xi_{i}}\varphi_{m}(\xi)\cdot\tau\not =0\text{ \ \emph{for some} }%
\tau\in T_{\xi_{i}}(\partial U) & \text{if }\xi_{i}\in\partial U,
\end{array}
\]
\emph{where} $T_{\xi_{i}}(\partial U)$ \emph{denotes the tangent space to}
$\partial U$ \emph{at the point} $\xi_{i}.$

\medskip We will show that, under these assumptions, $\varphi_{m}$ has a
critical point $\xi\in U^{m}$ with critical value $c_{0}.$ Moreover, the same
is true for any small enough ${\mathcal{C}}^{1}$-perturbation of $\varphi
_{m}.$ Property \textbf{P1)} is a common way of describing a change of
topology of the sublevel sets of $\varphi_{m}$ at the level $c_{0},$ and
$c_{0}$ is called a minimax value of $\varphi_{m}.$ It is a critical value if
$U^{m}$ is invariant under the negative gradient flow of $\varphi_{m}.$ If
this is not the case, we use property \textbf{P2) }to modify the gradient
vector field of $\varphi_{m}$ near $\partial U^{m}$ at the level $c_{0}$ and
thus obtain a new vector field with the same stationary points, and such that
$\overline{U}^{m}$ is invariant and $\varphi_{m}$ is a Lyapunov function for
the associated negative flow near the level $c_{0}$ (see Lemmas \ref{pertgrad}%
\ and \ref{def}\ below). This allows us to prove Theorem \ref{teo1}\ and the following.

\begin{teo}
\label{teo2} Let $m\geq1$ and assume that there exists an open subset $U$ of
$\Omega$ with smooth boundary, compactly contained in $\Omega,$ with $\inf
_{U}k>0,$ which satisfies \emph{P1)} and \emph{P2)}. Then, for $\rho$ small
enough, there exists a solution $u_{\rho}$ to Problem \eqref{In1} with
\[
\lim_{\rho\rightarrow0}\rho^{4}\int_{\Omega}k(x)e^{u_{\rho}}=64\,\pi^{2}\,m.
\]
Moreover, there is an $m$-tuple $(x_{1}^{\rho},\ldots,x_{m}^{\rho})\in U^{m}$,
such that as $\rho\rightarrow0$
\[
\nabla\varphi_{m}(x_{1}^{\rho},\ldots,x_{m}^{\rho})\rightarrow0,\quad
\varphi_{m}(x_{1}^{\rho},\ldots,x_{m}^{\rho})\rightarrow c_{0},\quad
\]
for which $u_{\rho}$ remains uniformly bounded on $\Omega\setminus\cup
_{j=1}^{m}B_{\delta}(x_{i}^{\rho})$, and
\[
\sup_{B_{\delta}(x_{i}^{\rho})}u_{\rho}\rightarrow+\infty,
\]
for any $\delta>0$.
\end{teo}

We will show that, for every $m\geq1,$ the set $U:=\{\xi\in\Omega
:{\mbox {dist}}(\xi,\partial\Omega)>\delta\}$ has property \textbf{P2)} at a
given $c_{0},$ for $\delta$ small enough (see Lemma \ref{bdryflow}). Thus, if
$\inf_{\Omega}k>0,$ and if there exist closed subsets $B_{0}\subset B$ of
$\Omega^{m}$ with
\[
\sup_{\xi\in B_{0}}\varphi_{m}(\xi)<\inf_{\gamma\in\Gamma}\sup_{\xi\in
B}\varphi_{m}(\gamma(\xi)),
\]
then both conditions \textbf{P1)} and \textbf{P2)} hold. Condition
\textbf{P1)}\ holds, for example, if $\varphi_{m}$ has a (possibly degenerate)
local minimum or local maximum. So a direct consequence of Theorem \ref{teo2}
is that in any bounded domain $\Omega$ with $\inf_{\Omega}k>0,$ Problem
(\ref{In1}) has at least one solution concentrating exactly at one point,
which corresponds to the minimum of the regular Green function $H$. Moreover
if, for example, $\Omega$ is a contractible domain obtained by joining
together $m$ disjoint bounded domains through thin enough tubes, then the
function $\varphi_{m}$ has a (possibly degenerate) local minimum, which gives
rise to a solution exhibiting $m$ points of concentration.

Finally, recall that Problem (\ref{In1}) corresponds to a standard case of
\emph{uniform singular convergence}, in the sense that the associated
nonlinear coefficient in Problem (\ref{In1}) --$\rho^{4}k(x)$-- goes to $0$
uniformly in $\bar{\Omega}$ as $\rho\rightarrow0$, property that is also
present in Problem (\ref{l}). Nontrivial topology strongly determines
existence of solutions. However, we expect that this strong influence should
decay under an inhomogeneous and \emph{non-uniform} singular behavior, where
critical points of an \emph{external} function determine existence and
multiplicity of solutions. See \cite{PM} for a recent two dimensional case of
this phenomenon.

\medskip The paper is organized as follows. Section $2$ is devoted to
describing a first approximation for the solution and to estimating the error.
Furthermore, Problem (\ref{In1}) is written as a fixed point problem,
involving a linear operator. In Section $3$ we study the invertibility of the
linear problem. In Section $4$ we solve a projected nonlinear problem. In
Section $5$ we show that solving the entire nonlinear problem reduces to
finding critical points of a certain functional. Section $6$ is devoted to the
proofs of Theorems \ref{teo1} and \ref{teo2}.

\section{Preliminaries and ansatz for the solution}

\bigskip

\noindent This section is devoted to construct a reasonably good approximation
$U$ for a solution of (\ref{In1}). The shape of this approximation will depend
on some points $\xi_{i}$, which we leave as parameters yet to be adjusted,
where the spikes are meant to take place. As we will see, a convenient set to
select $\xi=(\xi_{1},\ldots,\xi_{m})$ is
\begin{equation}
\mathcal{O}:=\Big\{\mathbf{\xi}\in\Omega^{m}:{\mbox
{dist}}(\xi_{j},\partial\Omega)\geq2\delta_{0},\ \forall\,j=1,\dots
,m,\hbox{ and }\min_{i\neq j}\abs{\xi_i - \xi_j}\geq2\delta_{0}%
\Big\} \label{Os}%
\end{equation}
where $\delta_{0}>0$ is a small fixed number. We thus fix $\xi\in\mathcal{O}$.

\medskip\noindent For numbers $\mu_{j}>0$, $j=1,\dots,m$, yet to be chosen,
$x\in\mathbb{R}^{4}$ and $\varepsilon>0$ we define
\begin{equation}
\label{Aa3}u_{j}(x)=4\log\frac{\mu_{j}(1+\varepsilon^{2})}{\mu_{j}%
^{2}\varepsilon^{2} + \abs{x-\xi_j}^{2} } - \log k(\xi_{j}),
\end{equation}
so that $u_{j}$ solves
\begin{equation}
\label{Aa4}\Delta^{2} u - \rho^{4} k(\xi_{j}) e^{u}=0 \hbox{ in } \mathbb{R}
^{4},
\end{equation}
with
\begin{equation}
\label{rhoeps}\rho^{4} = \frac{384 \, \varepsilon^{4}}{(1+\varepsilon^{2}%
)^{4}},
\end{equation}
that is, $\rho\sim\varepsilon$ as $\varepsilon\to0$.

\noindent Since $u_{j}$ and $\Delta u_{j}$ are not zero on the boundary
$\partial\Omega$, we will add to it a bi-harmonic correction so that the
boundary conditions are satisfied. Let $H_{j}(x)$ be the smooth solution of
\[
\label{Aa5}\quad%
\begin{cases}
\Delta^{2} H_{j}=0 & \hbox{ in } \Omega,\\
H_{j}=-u_{j} & \hbox{ on } \partial\Omega,\\
\Delta H_{j} = -\Delta u_{j} & \hbox{ on } \partial\Omega.
\end{cases}
\]
We define our first approximation $U(\xi)$ as
\begin{equation}
\label{Aa5a}U(\xi) \equiv\sum_{j=1}^{m} U_{j},\quad U_{j}\equiv u_{j} +
H_{j}\, .
\end{equation}
As we will rigorously prove below, $\left(  u_{j} + H_{j} \right)  (x) \sim
G(x,\xi_{j})$ where $G(x,\xi)$ is the Green function defined in (\ref{Un1}).

\noindent While $u_{j}$ is a good approximation to a solution of (\ref{In1})
near $\xi_{j}$, it is not so much the case for $U$, unless the remainder
$U\,-\, u_{j}\, = \, \big( H_{j} + \sum_{k\ne j} u_{k} \big)$ vanishes at main
order near $\xi_{j}$. This is achieved through the following precise choice of
the parameters $\mu_{k}$
\begin{equation}
\log\mu_{j}^{4} = \log k(\xi_{j}) + H(\xi_{j},\xi_{j}) + \sum_{i\neq j}%
G(\xi_{i},\xi_{j}). \label{muk}%
\end{equation}
We thus fix $\mu_{j}$ \emph{a priori} as a function of $\xi$. We write
\[
\mu_{j}=\mu_{j}(\xi)
\]
for all $j=1,\dots,m$. Since $\xi\in{\mathcal{O}}$,
\begin{equation}
\label{mu2}\frac1C\leq\mu_{j} \leq C, \quad\hbox{ for all } j=1,\dots,m,
\end{equation}
for some constant $C>0$.

\medskip The following lemma expands $U_{j}$ in $\Omega$.

\begin{lema}
\label{AaLe1} Assume $\xi\in\mathcal{O}$. Then we have
\begin{equation}
\label{Aa6}H_{j}(x)=H(x,\xi_{j})-4\log\mu_{j} (1+\varepsilon^{2}) + \log
k(\xi_{j}) + O(\mu_{j}^{2}\varepsilon^{2}),
\end{equation}
uniformly in $\Omega$, and
\begin{equation}
\label{Aa7}u_{j}(x)=4\log\mu_{j}(1+\varepsilon^{2}) -\log k(\xi_{j})
-8\log\abs{x-\xi_j} + O(\mu_{j}^{2} \varepsilon^{2}),
\end{equation}
uniformly in the region $\abs{x-\xi_j}\geq\delta_{0}$, so that in this
region,
\begin{equation}
\label{Aa8}U_{j}(x) = G(x,\xi_{j})+O(\mu_{j}^{2}\varepsilon^{2}).
\end{equation}

\end{lema}

\begin{proof}
Let us prove (\ref{Aa6}). Define $z(x)=H_{j}(x)+4\log\mu_{j}(1+\varepsilon
^{2}) -\log k(\xi_{j})- H(x,\xi_{j}) $. Then $z$ is a bi-harmonic function
which satisfies
\[
\label{Aa10}\quad%
\begin{cases}
\Delta^{2} z=0 & \hbox{ in } \Omega,\\
z=-u_{j} +4\log\mu_{j} (1+\varepsilon^{2}) -\log k(\xi_{j}) -8\log
\abs{\cdot-\xi_j} & \hbox{ on } \partial\Omega,\\
\Delta z = -\Delta u_{j} -\frac{16}{\abs{\cdot-\xi_j}^{2}} & \hbox{
on } \partial\Omega.
\end{cases}
\]

\noindent Let us define $w\equiv-\Delta z$. Thus $w$ is harmonic in $\Omega$
and
\[
\label{Aa10a}\sup_{\Omega} \abs{w}\leq\sup_{\partial\Omega}\abs{w}\leq
C\mu_{j}^{2}\varepsilon^{2}.
\]
We also have $\displaystyle{\sup_{\partial\Omega}\abs{z}\leq C\mu_{j}%
^{2}\varepsilon^{2}}$. Standard elliptic regularity implies
\[
\label{Aa10b}\sup_{\Omega} \abs{z}\leq C (\sup_{\Omega}\abs{w} +
\sup_{\partial\Omega}\abs{z})\leq C\mu_{j}^{2}\varepsilon^{2},
\]
as desired. The second estimate is direct from the definition of $u_{j}$.
\end{proof}

\bigskip\noindent Now, let us write
\begin{equation}
\Omega_{\varepsilon} = \varepsilon^{-1}\Omega, \quad\xi_{j}^{\prime
}=\varepsilon^{-1} \xi_{j}. \label{deltagamma}%
\end{equation}
Then $u$ solves (\ref{In1}) if and only if $v(y)\equiv u(\varepsilon
y)+4\log\rho\varepsilon$ satisfies
\begin{equation}
\label{Aa11}\quad%
\begin{cases}
\Delta^{2} v - k(\varepsilon y) e^{v} = 0, \quad\hbox{ in }\Omega
_{\varepsilon} ,\\
v = 4\log\rho\varepsilon, \quad\Delta v=0, \quad\hbox{ on } \partial
\Omega_{\varepsilon}.\\
\end{cases}
\end{equation}

\medskip\noindent Let us define $V(y)=U(\varepsilon y) + 4\log\rho\varepsilon
$, with $U$ our approximate solution (\ref{Aa5a}). We want to measure the size
of the error of approximation
\begin{equation}
R\equiv\Delta^{2} V - k(\varepsilon y)e^{V} . \label{error}%
\end{equation}
It is convenient to do so in terms of the following norm
\begin{equation}
\label{Aa12}\norm{v}_{*}=\sup_{y\in\Omega_{\varepsilon}} \Big| \Big[\sum
_{j=1}^{m} \frac{1}{(1 + \abs{y-\xi_j'}^{2})^{7/2}} +\varepsilon^{4}
\Big]^{-1}v(y)\Big|
\end{equation}

\noindent Here and in what follows, $C$ denotes a generic constant independent
of $\varepsilon$ and of $\xi\in\mathcal{O}$.

\begin{lema}
\label{AsmLe1} The error $R$ in $(\ref{error})$ satisfies
\[
\norm{R}_{*}\,\le\, C\,\varepsilon\quad\hbox{ as } \varepsilon\to0.
\]

\end{lema}

\begin{proof}
We assume first $\displaystyle{\abs{y-\xi_k'}<\frac{\delta_{0}}{\varepsilon}}
$, for some index $k$. We have
\begin{align*}
\label{Asm1}\Delta^{2} V(y) = \rho^{4}\sum_{j=1}^{m}
k(\xi_{j})e^{u_{j}( \varepsilon y)}= \frac{384 \,
\mu_{k}^{4}}{(\mu_{k}^{2} +\abs{y-\xi_k'}^{2} )^{4}} +
O(\varepsilon^{8}).
\end{align*}
Let us estimate $k(\varepsilon y)e^{V(y)}$. By (\ref{Aa6}) and the definition
of $\mu_{j}^{\prime}s$,
\begin{align*}
H_{k}(x)  &  = H(\xi_{k}, \xi_{k}) - 4\log\mu_{k} + \log k(\xi_{j}) +
O(\mu_{k}^{2} \varepsilon^{2}) + O(\abs{x-\xi_k})\\
&  = -\sum_{j\neq k} G(\xi_{j},\xi_{k})+ O(\mu_{k}^{2}\varepsilon^{2}) +
O(\abs{x-\xi_k}),
\end{align*}
\label{Asm2} and if $j\neq k$, by (\ref{Aa8})
\[
\label{Asm3a}U_{j}(x) = u_{j}(x) + H_{j}(x) = G(\xi_{j},\xi_{k}) +
O(\abs{x-\xi_k}) + O(\mu_{j}^{2} \varepsilon^{2}).
\]
Then
\begin{equation}
\label{Asm3b}H_{k}(x)+\sum_{j\neq k}U_{j}(x)= O(\varepsilon^{2}) +
O(\abs{x-\xi_k}).
\end{equation}
Therefore,
\begin{align*}
\label{Asm3}k(\varepsilon y)e^{V(y)}  &  = k(\varepsilon y)\varepsilon^{4}
\rho^{4} \exp\Big\{u_{k}(\varepsilon y) + H_{k}(\varepsilon y) + \sum_{j\neq
k} U_{j}(\varepsilon y)\Big\}\\
&  = \frac{384\mu_{k}^{4}k(\varepsilon y)}{( \mu_{k}^{2} + \abs{y-\xi_k'}^{2}
)^{4}k(\xi_{k})} \Big\{ 1 + O(\varepsilon\abs{y-\xi_k'}) + O(\varepsilon^{2})
\Big\}\\
&  = \frac{384\mu_{k}^{4}}{(\mu_{k}^{2} + \abs{y-\xi_k'}^{2} )^{4}}\;\Big\{ 1
+ O(\varepsilon\abs{y-\xi_k'}) \Big\}
\end{align*}
We can conclude that in this region
\[
|R(y)| \leq C \frac{\varepsilon\abs{y-\xi_k'}}{ (1 +
\abs{y-\xi_k'}^{2} )^{4}} + O(\varepsilon^{4}).
\]
If $\displaystyle{\abs{y-\xi_j'}\geq\frac{\delta_{0}}{\varepsilon}} $ for all
$j$, using (\ref{Aa6}), (\ref{Aa7}) and (\ref{Aa8}) we obtain
\[
\Delta^{2} V = O(\varepsilon^{4} \rho^4)\quad\hbox{ and } \quad
k(\varepsilon y)e^{V(y)} = O(\varepsilon^{4} \rho^4).
\]
Hence, in this region,
\[
R(y) = O(\varepsilon^{8})
\]
so that finally
\begin{align}
\label{Asm6}\norm{R}_{*} = O(\varepsilon).\nonumber
\end{align}

\end{proof}

\medskip\noindent Next we consider the energy functional associated with
(\ref{In1})
\begin{equation}
\label{In4}J_{\rho}[u]=\frac12 \int_{\Omega} (\Delta u)^{2} - \rho^{4}
\int_{\Omega} k(x)e^{u}, \quad u\in H^{2}(\Omega)\cap H_{0}^{1}(\Omega).
\end{equation}
We will give an asymptotic estimate of $J_{\rho}[U]$, where $U(\xi)$ is the
approximation (\ref{Aa5a}). Instead of $\rho$, we use the parameter
$\varepsilon$ (defined in (\ref{rhoeps})) to obtain the following expansion:

\begin{lema}
\label{Ene1} With the election of $\mu_{j} $'s given by (\ref{muk}),
\begin{equation}
\label{AesTeo1}J_{\rho}[U]=-128\, \pi^{2} \, m + 256 \, \pi^{2} \, m
\abs{\log \ve} + 32 \, \pi^{2} \, \varphi_{m}(\xi) + \varepsilon
\Theta_{\varepsilon}(\xi),
\end{equation}
where $\Theta_{\varepsilon}(\xi)$ is uniformly bounded together with its
derivatives if $\xi\in\mathcal{O}$, and $\varphi_{m}$ is the function defined
in (\ref{Rob}).
\end{lema}

\begin{proof}
We have
\begin{align*}
J_{\rho}[U]  &  = \frac12 \sum_{j=1}^{m} \int_{\Omega} (\Delta U_{j})^{2} +
\frac12 \sum_{j\neq i}\int_{\Omega} \Delta U_{j}\,\Delta U_{i} -\rho^{4}%
\int_{\Omega} k(x) e^{U}\\
&  \equiv I_{1} + I_{2} + I_{3};
\end{align*}
Note that $\Delta^{2} U_{j}=\Delta^{2} u_{j} = \rho^{4} k(\xi_{j})e^{u_{j}}$
in $\Omega$ and $U_{j}=\Delta U_{j}=0$ in $\partial\Omega$. Then
\[
I_{1}= \frac12 \rho^{4} \sum_{j=1}^{m} k(\xi_{j})\int_{\Omega} e^{u_{j}}U_{j}
\; \hbox{ and }\; I_{2}= \frac12 \rho^{4} \sum_{j\neq i}k(\xi_{j})
\int_{\Omega} e^{u_{j}}U_{i}.
\]
Let us define the change of variables $x=\xi_{j}+\mu_{j}\varepsilon y$, where
$x\in\Omega$ and $y\in\Omega_{j}\equiv(\mu_{j} \varepsilon)^{-1}(\Omega
-\xi_{j})$. Using Lemma~\ref{AaLe1} and the definition of $\rho$ in terms of
$\varepsilon$ in (\ref{rhoeps}) we obtain
\begin{align*}
I_{1}  &  = 192 \sum_{j=1}^{m} \int_{\Omega_{j}} \frac{1}{(1+\abs{y}^{2})^{4}}
\left\{  4\log\frac{1}{1+\abs{y}^{2}}-8\log\mu_{j}\varepsilon+ H(\xi_{j}%
,\xi_{j})+ O(\mu_{j}\varepsilon\abs{y})\right\} \\
&  = 32\pi^{2} \sum_{j=1}^{m} \Big\{ H(\xi_{j},\xi_{j}) -8\log\mu
_{j}\varepsilon\Big\} -64 \, \pi^{2} m + O\left(  \varepsilon\mu_{j}
\int_{\Omega_{j}} \frac{\abs{y}}{(1+\abs{y}^{2})^{4}}\right) \\
&  = 32\pi^{2} \sum_{j=1}^{m} \Big\{ H(\xi_{j},\xi_{j}) -8\log\mu
_{j}\varepsilon\Big\} -64 \, \pi^{2} m + \varepsilon\Theta(\xi),
\end{align*}
where $\Theta_{\varepsilon}(\xi)$ is bounded together with its derivates if
$\xi\in\mathcal{O}$. Besides we have used the explicit values
$\displaystyle{\int_{\mathbb{R}^{4}}\frac{1}{(1+\abs{y}^{2})^{4}}=\frac
{\pi^{2}}{6}}$, and $\displaystyle{\int_{\mathbb{R}^{4}}\frac{\log
(1+\abs{y}^{2})}{(1+\abs{y}^{2})^{4}}=\frac{\pi^{2}}{12}}$.

\medskip\noindent We consider now $I_{2}$. As above,
\begin{align*}
\frac12\rho^{4} \int_{\Omega} e^{u_{j}}U_{i}  &  = \int_{\Omega_{j}}
\frac{192}{( 1 + \abs{y}^{2} )^{4}} \Big\{ u_{i}(\xi_{j}+\mu_{j}\varepsilon y)
+ H_{i}(\xi_{j} + \mu_{j}\varepsilon y) \Big\}\\
&  = \int_{\Omega_{j}} \frac{192}{( 1 + \abs{y}^{2})^{4}} \Big\{ u_{i}(
\xi_{j} + \mu_{j}\varepsilon y) -4\log\mu_{i}(1+\varepsilon^{2}) + \log
k(\xi_{i})+8\log\abs{\xi_j -\xi_i}\Big\}\\
&  + \int_{\Omega_{j}} \frac{192}{( 1 + \abs{y}^{2})^{4}} \Big\{ H_{i}(\xi_{j}
+ \mu_{j} \varepsilon y) - H_{i}(\xi_{j}) \Big\}\\
&  + \int_{\Omega_{j}} \frac{192}{( 1 + \abs{y}^{2})^{4}} \Big\{ H_{i}(\xi
_{j}) - H(\xi_{j},\xi_{i})+4\log\mu_{i}(1+\varepsilon^{2})-\log k(\xi_{i})
\Big\}\\
&  + G(\xi_{j},\xi_{i})\int_{\Omega_{j}} \frac{192}{( 1 + \abs{y}^{2} )^{4}}\\
&  = 32\pi^{2} G(\xi_{i},\xi_{j}) + O\left(  \varepsilon\mu_{j} \int
_{\Omega_{j}} \frac{\abs{y}}{( 1 + \abs{y}^{2})^{4}} \right)  + O(\mu_{j}%
^{2}\varepsilon^{2})\\
&  = 32\pi^{2} G(\xi_{i},\xi_{j}) + \varepsilon\Theta_{\varepsilon}(\xi).
\end{align*}
Thus
\begin{equation}
\label{I2}I_{2}=32\pi^{2}\sum_{j\neq i} G(\xi_{i},\xi_{j}) + \varepsilon
\Theta_{\varepsilon}(\xi).
\end{equation}
Finally we consider $I_{3}$. Let us denote $A_{j}\equiv B(\xi_{j},\delta_{0})$
and $x=\xi_{j} + \mu_{j}\varepsilon y$. Then using again Lemma~\ref{AaLe1}
\begin{align*}
I_{3}  &  = -\rho^{4} \sum_{j=1}^{m} \int_{A_{j}} k(x)e^{U} + O(\varepsilon
^{4})\\
&  = -\rho^{4} \sum_{j=1}^{m} \int_{B(0,\frac{\delta_{0}}{\mu_{j}\varepsilon
})}\frac{k(\xi_{j} + \mu_{j}\varepsilon y)}{k(\xi_{j})(1+\abs{y}^{2})^{4}}
\frac{(1+\varepsilon^{2})^{4}}{\varepsilon^{4}} \left(  1+ O(\varepsilon
\mu_{j}\abs{y})\right)  + O(\varepsilon^{4})\\
&  = -384 \, m \int_{\mathbb{R}^{4}}\frac{1}{(1+\abs{y}^{2})^{4}} + O\left(
\varepsilon\mu_{j}\int_{\mathbb{R}^{4}}\frac{\abs{y}}{(1+\abs{y}^{2})^{4}%
}\right) \\
&  = -64\pi^{2} m + \varepsilon\Theta_{\varepsilon}(\xi),
\end{align*}
uniformly in $\xi\in\mathcal{O}$. Thus, we can conclude the following
expansion of $J_{\rho}[U]$:
\begin{equation}
\label{Ener}J_{\rho}[U]=-128 \, m \, \pi^{2} + 256 \, m \, \pi^{2}
\abs{\log \ve} + 32 \, \pi^{2} \, \varphi_{m}(\xi) + \varepsilon
\Theta_{\varepsilon}(\xi),
\end{equation}
where $\Theta_{\varepsilon}(\xi)$ is a bounded function together with is
derivates in the region $\xi\in\mathcal{O}$, $\varphi_{m}$ defined as in
(\ref{Rob}) and $\displaystyle{\rho^{4} = \frac{384 \varepsilon^{4}%
}{(1+\varepsilon^{2})^{4}}}$.
\end{proof}

\medskip\noindent In the subsequent analysis we will stay in the expanded
variable $y\in\Omega_{\varepsilon}$ so that we will look for solutions of
problem (\ref{Aa11}) in the form $v=V +\psi$, where $\psi$ will represent a
lower order correction. In terms of $\psi$, problem (\ref{Aa11}) now reads
\begin{equation}
\label{Asm7}\quad%
\begin{cases}
\mathcal{L}_{\varepsilon}(\psi) \equiv\Delta^{2} \psi- W \psi= -R + N(\psi) &
\hbox{ in } \Omega_{\varepsilon},\\
\psi= \Delta\psi= 0 & \hbox{ on } \partial\Omega_{\varepsilon},
\end{cases}
\end{equation}
where
\begin{equation}
\label{neerre}N(\psi) = W[e^{\psi}- \psi- 1] \quad{\mbox {and}} \quad
W=k(\varepsilon y)e^{V}.
\end{equation}
\medskip\noindent Note that
\begin{equation}
\label{AsmLe2a}W(y)=\sum_{j=1}^{m} \frac{384 \, \mu_{j}^{4}}{(\mu_{j}^{2} +
\abs{y-\xi_j'}^{2})^{4}}(1+O(\varepsilon\abs{y-\xi_j'})) \quad\hbox{for }
y\in\Omega_{\varepsilon}.
\end{equation}
This fact, together with the definition of $N(\psi)$ given in \eqref{neerre},
give the validity of the following

\begin{lema}
\label{AsmLe2} For $\xi\in\mathcal{O}$, $\norm{W}_{*}=O(1)$ and
$\norm{N(\psi)}_{*}=O(\norm{\psi}^{2}_{\infty})$ as $\norm{\psi}_{\infty}
\to0$.
\end{lema}

\bigskip


\section{The linearized problem}

\label{Lp}

\noindent In this section we develop a solvability theory for the fourth-order
linear operator $\mathcal{L}_{\varepsilon}$ defined in (\ref{Asm7}) under
suitable orthogonality conditions. We consider
\begin{equation}
\label{Lp1}\mathcal{L}_{\varepsilon}(\psi)\equiv\Delta^{2} \psi- W(y)\psi,
\end{equation}
where $W(y)$ was introduced in (\ref{Asm7}). By expression (\ref{AsmLe2a}) and
setting $z= y-\xi_{j}^{\prime}$, one can easily see that formally the operator
$\mathcal{L}_{\varepsilon}$ approaches, as $\varepsilon\to0 $,
the operator in $\mathbb{R}^{4}$
\begin{equation}
\label{Lp3}\mathcal{L}_{j}(\psi)\equiv\Delta^{2} \psi- \frac{384\, \mu_{j}%
^{4}}{(\mu_{j}^{2}+\abs{z}^{2})^{4}}\psi,
\end{equation}
namely, equation $\Delta^{2} v - e^{v} =0$ linearized around the radial
solution $v_{j}(z)=\log\frac{384\mu_{j}^{4}}{(\mu_{j}^{2} + \abs{z}^{2})^{4}}%
$. Thus the key point to develop a satisfactory solvability theory for the
operator $\mathcal{L}_{\varepsilon}$ is the non-degeneracy of $v_{j}$ up to
the natural invariances of the equation under translations and dilations. In
fact, if we set
\begin{align}
\label{Lp4}Y_{0j}(z)  &  = 4\frac{\abs{z}^{2}-\mu_{j}^{2}}{\abs{z}^{2} +
\mu_{j}^{2}},\\
Y_{ij}(z)  &  = \frac{8z_{i}}{ \mu_{j}^{2} +\abs{z}^{2}}, \quad i=1,\dots,4,
\end{align}
the only bounded solutions of $\mathcal{L}_{j}(\psi)=0$ in $\mathbb{R}^{4}$
are linear combinations of $Y_{ij}$, $i=0,\dots,4$; see Lemma $3.1$ in
\cite{BaPa} for a proof.

\medskip\noindent We define for $i=0,\dots,4$ and $j=1,\dots,m$,
\[
Z_{ij}(y)\equiv Y_{ij}\left(  y-\xi_{j}^{\prime}\right)  ,\; i=0,\dots,4.
\]
Additionally, let us consider $R_{0}$ a large but fixed number and $\chi$ a
radial and smooth cut-off function with $\chi\equiv1$ in $B(0,R_{0})$ and
$\chi\equiv0$ in $\mathbb{R}^{4} \setminus B(0,R_{0}+1)$. Let
\[
\chi_{j}(y)=\chi(\abs{y-\xi_j'}),\quad j=1,\ldots,m.
\]
Given $h\in L^{\infty}(\Omega_{\varepsilon})$, we consider the problem of
finding a function $\psi$ such that for certain scalars $c_{ij}$ one has
\begin{align}
\label{Lp5}%
\begin{cases}
& \mathcal{L}_{\varepsilon}(\psi)= h + \sum_{i=1}^{4} \sum_{j=1}^{m}
c_{ij}\chi_{j}Z_{ij}, \quad\hbox{ in } \Omega_{\varepsilon},\\
& \psi=\Delta\psi= 0, \quad\hbox{ on } \partial\Omega_{\varepsilon},\\
& \int_{\Omega_{\varepsilon}}\chi_{j}Z_{ij}\psi=0, \hbox{ for all }
i=1,\dots,4,\; j=1,\dots, m.
\end{cases}
\end{align}
We will establish a priori estimates for this problem. To this end we shall
introduce an adapted norm in $\Omega_{\varepsilon}$, which has been introduced
previously in \cite{KMP1}. Given $\psi:\Omega_{\varepsilon}\to\mathbb{R}$ and
$\alpha\in\mathbb{N}^{m}$ we define
\begin{equation}
\label{norm2}\norm{\psi}_{**}\equiv\sum_{j=1}^{m}\norm{\psi}_{C^{4,\alpha
}(r_{j}<2)} + \sum_{j=1}^{m} \sum_{\abs{\alpha} \leq3}%
\norm{r_j^{\abs{\alpha}} D^{\alpha}\psi}_{L^{\infty}(r_{j}\geq2)},
\end{equation}
with $r_{j}=\abs{y-\xi_j'}$.

\begin{prop}
\label{LpTeo1} There exist positive constants $\varepsilon_{0}>0$ and $C>0$
such that for any $h\in L^{\infty}(\Omega_{\varepsilon})$, with $\Vert
h\Vert_{\ast}<\infty$, and any $\xi\in\mathcal{O}$, there is a unique solution
$\psi=T(h)$ to Problem (\ref{Lp5}) for all $\varepsilon\leq\varepsilon_{0}$,
which defines a linear operator of $h$. Besides, we have the estimate
\begin{equation}
\norm{T(h)}_{\ast\ast}\leq C\,\abs{\log \ve}\,\norm{h}_{\ast}. \label{LpTeo1a}%
\end{equation}

\end{prop}

\noindent The proof will be split into a serie of lemmas which we state and
prove next. The first step is to obtain a priori estimates for the problem
\begin{align}
\label{Lp19}%
\begin{cases}
& \mathcal{L}_{\varepsilon}(\psi)= h, \quad\hbox{ in } \Omega_{\varepsilon},\\
& \psi=\Delta\psi= 0, \quad\hbox{ on } \partial\Omega_{\varepsilon},\\
& \int_{\Omega_{\varepsilon}}\chi_{j}Z_{ij}\psi=0, \hbox{ for all }
i=0,\dots,4,\; j=1,\dots, m.
\end{cases}
\end{align}
which involves more orthogonality conditions than those in (\ref{Lp5}).
\noindent We have the following estimate.

\begin{lema}
\label{LpLe3} There exist positive constants $\varepsilon_{0}>0$ and $C>0$
such that for any $\psi$ solution of Problem $(\ref{Lp19})$ with $h \in
L^{\infty}(\Omega_{\varepsilon}) $, $\| h \|_{*} <\infty$, and $\xi
\in\mathcal{O}$, then
\begin{equation}
\norm{\psi}_{**}\leq C\, \norm{h}_{*}%
\end{equation}
for all $\varepsilon\in(0, \varepsilon_{0} )$.
\end{lema}

\begin{proof}
We carry out the proof by a contradiction argument. If the above fact were
false, then, there would exist a sequence $\varepsilon_{n}\to0$, points
$\xi^{n} = (\xi_{1}^{n} , \ldots, \xi_{m}^{n} ) \in\mathcal{O}$, functions
$h_{n}$ with $\norm{h_n}_{*}\to0$ and associated solutions $\psi_{n}$ with
$\norm{\psi_n}_{**}=1$ such that
\begin{align}
\label{Lp20}%
\begin{cases}
& \mathcal{L}_{\varepsilon_{n}}(\psi_{n})= h_{n}, \quad\hbox{ in }
\Omega_{\varepsilon_{n}},\\
& \psi_{n}=\Delta\psi_{n} = 0, \quad\hbox{ on } \partial\Omega_{\varepsilon
_{n}},\\
& \int_{\Omega_{\varepsilon_{n}}}\chi_{j}Z_{ij}\psi_{n} =0, \hbox{
for all } i=0,\dots,4,\; j=1,\dots, m.
\end{cases}
\end{align}
Let us set $\tilde{\psi}_{n}(x)=\psi_{n}(x/\varepsilon_{n})$, $x\in\Omega$. It
is directly checked that for any $\delta^{\prime}>0$ sufficiently small
$\tilde{\psi}_{n}$ solves the problem
\begin{align*}
\label{Lp20a}%
\begin{cases}
& \Delta^{2}\tilde{\psi}_{n}= O(\varepsilon_{n}^{4}) + \varepsilon_{n}%
^{-4}h_{n}=o(1), \quad\hbox{ uniformly in } \Omega\backslash\cup_{k=1}^{m}
B(\xi_{j}^{n},\delta^{\prime}),\\
& \tilde{\psi}_{n}=\Delta\tilde{\psi}_{n} = 0, \quad\hbox{ on } \partial
\Omega,\\
&
\end{cases}
\end{align*}
together with $\norm{\tilde{\psi}_n}_{\infty}\leq1$ and
$\norm{\Delta\tilde{\psi}_n}_{\infty} \leq C_{\delta^{\prime}}$, in the
considered region. Passing to a subsequence, we then get that $\xi^{n}\to
\xi^{*} \in\mathcal{O}$ and $\tilde{\psi}_{n} \to0$ in the $C^{3,\alpha}$
sense over compact subsets of $\Omega\backslash\{ \xi_{1}^{*},\dots,\xi
_{m}^{*}\}$. In particular
\[
\sum_{\abs{\alpha}\leq3}\frac{1}{\varepsilon_{n}^{\abs{\alpha}}}%
\abs{D^{\alpha}\psi_n(y)} \to0,\quad\hbox{ uniformly in }
\abs{y-(\xi_j^n)'}\geq\frac{\delta^{\prime}}{2\varepsilon_{n}},
\]
for any $\delta^{\prime}>0$ and $j\in\{1,\dots,m\}$. We obtain thus that
\begin{equation}
\label{conv1}\sum_{j=1}^{m} \sum_{\abs{\alpha}\leq3}%
\norm{r_j^{\abs{\alpha}} D^{\alpha}\psi_n}_{L^{\infty}(r_{j}\geq\delta
^{\prime}/\varepsilon_{n} )},\to0,
\end{equation}
for any $\delta^{\prime}>0$. In conclusion, the \emph{exterior portion} of
$\norm{\psi_n}_{**}$ goes to zero, see (\ref{norm2}).

\bigskip\noindent Let us consider now a smooth radial cut-off function
$\hat{\eta}$ with $\hat{\eta}(s)=1$ if $s<\frac12$, $\hat{\eta}(s)=0$ if
$s\geq1$, and define
\[
\hat{\psi}_{n,j}(y)=\hat{\eta}_{j}(y)\psi_{n}(y)\equiv\hat{\eta}%
\Big(\frac{\varepsilon_{n}}{\delta_{0}}\abs{y-(\xi_j^n)'} \Big)\psi_{n}(y),
\]
such that
\[
\operatorname{supp} \hat{\psi}_{n,j} \subseteq B\Big((\xi_{j}^{n})^{\prime},
\frac{\delta_{0}}{\varepsilon_{n}} \Big).
\]
We observe that
\[
\mathcal{L}_{\varepsilon_{n}}(\hat{\psi}_{n,j})=\hat{\eta}_{j}h_{n} +
F(\hat{\eta}_{j},\psi_{n}),
\]
where
\[
F(f,g)=g\Delta^{2} f + 2\Delta f \Delta g + 4\nabla(\Delta f)\cdot\nabla g + 4
\nabla f\cdot\nabla(\Delta g)
\]
\begin{equation}
\label{F1}+ 4\sum_{i,j=1}^{4} \frac{\partial^{2} f}{\partial y_{i} \partial
y_{j}} \, \frac{\partial^{2} g}{\partial y_{i} \partial y_{j}}.
\end{equation}
Thus we get
\begin{align}
\label{Inter1}%
\begin{cases}
& \Delta^{2} \hat{\psi}_{n,j} = W_{n}(y)\hat{\psi}_{n,j} + \hat{\eta}_{j}h_{n}
+ F(\hat{\eta}_{j},\psi_{n}) \quad\hbox{ in } B\left(  (\xi_{j}^{n})^{\prime},
\frac{\delta_{0}}{\varepsilon_{n}} \right)  ,\\
& \hat{\psi}_{n,j}=\Delta\hat{\psi}_{n,j}=0 \quad\hbox{ on } \partial B\left(
(\xi_{j}^{n})^{\prime}, \frac{\delta_{0}}{\varepsilon_{n}} \right)  .
\end{cases}
\end{align}
The following intermediate result provides an outer estimate. For notational
simplicity \emph{we omit} the subscript $n$ in the quantities involved.

\begin{lema}
\label{LpLe3a} There exist constants $C,R_{0}>0$ such that for large $n$
\begin{equation}
\label{LpLe3aa}\sum_{\abs{\alpha}\leq3}%
\norm{r_j^{\abs{\alpha}} D^{\alpha}\hat{\psi}_j}_{L^{\infty}(r_{j}\geq R_{0}%
)}\leq C\, \{\norm{\hat{\psi}_j}_{L^{\infty}(r_{j}< 2R_{0})} + o(1) \}.
\end{equation}

\end{lema}

\begin{proof}
We estimate the righ-hand side of (\ref{Inter1}). If $2<r_{j}<\delta
_{0}/\varepsilon$ we get
\[
\Delta^{2}\hat{\psi}_{j}=O\Big(\frac{1}{r_{j}^{8}}\Big)\hat{\psi}_{j}+\frac
{1}{r_{j}^{7}}o(1)+O(\varepsilon^{4})+O\Big(\frac{\varepsilon^{3}}{r_{j}%
}\Big)+O\Big(\frac{\varepsilon^{2}}{r_{j}^{2}}\Big)+O\Big(\frac{\varepsilon
}{r_{j}^{3}}\Big).
\]
From (\ref{Inter1}) and standard elliptic estimates we have
\[
\sum_{\abs{\alpha}\leq3}\abs{D^{\alpha}\hat{\psi}_j}\leq C\,\Big\{\frac
{1}{r_{j}^{8}}\norm{\hat{\psi}_j}_{L^{\infty}(r_{j}>1)}+\frac{1}{r_{j}^{7}%
}o(1)+O\Big(\frac{\varepsilon}{r_{j}^{3}}\Big)\Big\},\quad\hbox{ in }2\leq
r_{j}\leq\frac{\delta_{0}}{\varepsilon}.
\]
Now, if $r_{j}\geq2$
\[
\abs{r_j^{\abs{\alpha}}D^{\alpha}\hat{\psi}_j}\leq C\;\Big\{\frac{1}{r_{j}%
^{5}}\norm{\hat{\psi}_j}_{L^{\infty}(r_{j}>1)}+o(1)\Big\},\quad
\abs{\alpha}\leq3.
\]
Finally
\[
\frac{1}{r_{j}^{5}}\norm{\hat{\psi}_j}_{L^{\infty}(r_{j}>1)}\leq
\norm{\hat{\psi}_j}_{L^{\infty}(1<r_{j}<R_{0})}+\frac{1}{R_{0}^{5}%
}\norm{\hat{\psi}_j}_{L^{\infty}(r_{j}>R_{0})},
\]
thus fixing $R_{0}$ large enough we have
\[
\sum_{\abs{\alpha}\leq3}%
\norm{r_j^{\abs{\alpha}}D^{\alpha}\hat{\psi}_j}_{L^{\infty}(r_{j}\geq R_{0}%
)}\leq C\,\{\norm{\hat{\psi}_j}_{L^{\infty}(1<r_{j}<R_{0})}+o(1)\},\quad
2<r_{j}<\frac{\delta_{0}}{\varepsilon},
\]
and then (\ref{LpLe3aa}).
\end{proof}

\noindent We continue with the proof of Lemma~\ref{LpLe3}.

\noindent Since $\norm{\psi_n}_{**}=1$ and using (\ref{conv1}) and
Lemma~\ref{LpLe3a} we have that there exists an index $j\in\{ 1,\dots,m\}$
such that
\begin{equation}
\label{inf1}\liminf_{n\to\infty} \norm{\psi_n}_{L^{\infty}(r_{j}<R_{0})}%
\geq\alpha>0.
\end{equation}
Let us set $\tilde{\psi}_{n}(z)=\psi_{n}((\xi_{j}^{n})^{\prime}+ z)$. We
notice that $\tilde{\psi}_{n}$ satisfies
\[
\Delta^{2} \tilde{\psi}_{n} - W((\xi_{j}^{n})^{\prime}+ z) \, \tilde{\psi}_{n}
= h_{n}((\xi_{j}^{n})^{\prime}+ z), \quad\hbox{ in
} \Omega_{n} \equiv\Omega_{\varepsilon}-(\xi_{j}^{n})^{\prime}\ .
\]
Since $\psi_{n}$, $\Delta\psi_{n}$ are bounded uniformly, standard elliptic
estimates allow us to assume that $\tilde{\psi}_{n}$ converges uniformly over
compact subsets of $\mathbb{R}^{4}$ to a bounded, non-zero solution
$\tilde{\psi}$ of
\[
\Delta^{2} \psi- \frac{384\mu_{j}^{4}}{(\mu_{j}^{2}+\abs{z}^{2})^{4}}\psi=0.
\]
This implies that $\tilde{\psi}$ is a linear combination of the functions
$Y_{ij},\,i=0,\dots,4$. But orthogonality conditions over $\tilde{\psi}_{n}$
pass to the limit thanks to $\norm{\tilde{\psi}_n}_{\infty}\leq1$ and
dominated convergence. Thus $\tilde{\psi}\equiv0$, a contradiction with
(\ref{inf1}). This conclude the proof.
\end{proof}

\medskip\noindent Now we will deal with problem (\ref{Lp19}) lifting the
orthogonality constraints $\int_{\Omega_{\varepsilon}}\chi_{j}Z_{0j}%
\psi=0,\quad j=1,\dots,m$, namely
\begin{align}
\label{Lp21}%
\begin{cases}
& \mathcal{L}_{\varepsilon}(\psi)= h, \quad\hbox{ in } \Omega_{\varepsilon},\\
& \psi= \Delta\psi= 0, \quad\hbox{ on } \partial\Omega_{\varepsilon},\\
& \int_{\Omega_{\varepsilon}}\chi_{j}Z_{ij}\psi=0, \hbox{ for all }
i=1,\dots,4\; j=1,\dots, m.
\end{cases}
\end{align}
We have the following a priori estimates for this problem.


\begin{lema}
\label{LpLe4} There exist positive constants $\varepsilon_{0}$ and $C$ such
that, if $\psi$ is a solution of (\ref{Lp21}), with $h \in L^{\infty}%
(\Omega_{\varepsilon})$, $\| h \|_{*} <\infty$ and with $\xi\in\mathcal{O}$,
then
\begin{equation}
\label{LpLe4a}\norm{\psi}_{**} \leq C\, \abs{\log \ve}\, \norm{h}_{*}%
\end{equation}
for all $\varepsilon\in(0, \varepsilon_{0})$.
\end{lema}

\begin{proof}
Let $R>R_{0}+1$ be a large and fixed number. Let us consider $\hat{Z}_{0j}$ be
the following function
\begin{equation}
\label{Lp22a}\hat{Z}_{0j}(y)=Z_{0j}(y)-1 + a_{0j}G(\varepsilon y,\xi_{j}),
\end{equation}
where $a_{0j}=(H(\xi_{j},\xi_{j})-8\log(\varepsilon R))^{-1}$. It is clear
that if $\varepsilon$ is small enough
\begin{align}
\label{Lp22b}\hat{Z}_{0j}(y)  &  = Z_{0j}(y) + a_{0j}\left(  G(\varepsilon y,
\xi_{j}) - H(\xi_{j},\xi_{j}) + 8\log(\varepsilon R)\right) \nonumber\\
&  = Z_{0j}(y) + \frac{1}{\abs{\log \ve}}\left(  O(\varepsilon r_{j}) +
8\log\frac{R}{r_{j}}\right)  .
\end{align}
and $Z_{0j}(y)=O(1)$. Next we consider radial smooth cut-off functions
$\eta_{1}$ and $\eta_{2}$ with the following properties:
\begin{align*}
&  0\leq\eta_{1}\leq1, \quad\eta_{1}\equiv1 \hbox{ in } B(0,R), \quad\eta
_{1}\equiv0 \hbox{ in } \mathbb{R}^{4} \setminus B(0,R+1),\quad\hbox{ and }\\
&  0\leq\eta_{2}\leq1, \quad\eta_{2}\equiv1 \hbox{ in } B\left(
0,\frac{\delta_{0}}{3\varepsilon} \right)  , \quad\eta_{2}\equiv0 \hbox{ in }
R^{4} \setminus B\left(  0,\frac{\delta_{0}}{2\varepsilon}\right)  .
\end{align*}
Then we set
\begin{equation}
\label{Lp23}\eta_{1j}(y)=\eta_{1} (r_{j}),\quad\eta_{2j}(y)=\eta_{2} (r_{j}),
\end{equation}
and define the test function
\[
\label{Lp24}\tilde{Z}_{0j}=\eta_{1j}Z_{0j} + (1-\eta_{1j})\eta_{2j}\hat
{Z}_{0j}.
\]
Note the $\tilde{Z}_{0j}$'s behavior throught $\Omega_{\varepsilon}$
\begin{align}
\label{Lp24_1}\tilde{Z}_{0j}=
\begin{cases}
Z_{0j}, & r_{j}\leq R\\
\eta_{1j}(Z_{0j}-\hat{Z}_{0j}) + \hat{Z}_{0j}, & R < r_{j}\leq R + 1\\
\hat{Z}_{0j}, & R+1 < r_{j} \leq\frac{\delta_{0}}{3\varepsilon}\\
\eta_{2j} \hat{Z}_{0j}, & \frac{\delta_{0}}{3\varepsilon} < r_{j} \leq
\frac{\delta_{0}}{2\varepsilon}\\
0 & otherwise.
\end{cases}
\end{align}
In the subsequent, we will label these four regions as
\[
\Omega_{0}\equiv\left\{  r_{j}\leq R \right\}  ,\quad\Omega_{1}\equiv\left\{
R < r_{j}\leq R+1\right\}  , \quad\Omega_{2}\equiv\left\{  R+1 < r_{j}
\leq\frac{\delta_{0}}{3\varepsilon}\right\}  ,
\]
\[
\hbox{ and }\quad\Omega_{3}\equiv\left\{  \frac{\delta_{0}}{3\varepsilon} <
r_{j}\leq\frac{\delta_{0}}{2\varepsilon}\right\}  .
\]
Let $\psi$ be a solution to problem (\ref{Lp21}). We will modify $\psi$ so
that the extra orthogonality conditions with respect to $Z_{0j}$'s hold. We
set
\begin{equation}
\label{Lp240}\tilde{\psi}=\psi+ \sum_{j=1}^{m} d_{j} \tilde{Z}_{0j}.
\end{equation}
We adjust the constants $d_{j}$ so that
\begin{equation}
\label{Lp24a}\int_{\Omega_{\varepsilon}} \chi_{j}Z_{ij}\tilde{\psi}=0,
\quad\hbox{ for all } i=0,\dots,4; \; j=1,\dots,m.
\end{equation}
Then,
\begin{equation}
\label{Lp25}\mathcal{L}_{\varepsilon}(\tilde{\psi}) = h + \sum_{j=1}^{m}
d_{j}\mathcal{L}_{\varepsilon}(\tilde{Z}_{0j})\, .
\end{equation}
If (\ref{Lp24a}) holds, the previous lemma allows us to conclude
\begin{equation}
\label{Lp26}\norm{\tilde{\psi}}_{**} \leq C \Big\{ \norm{h}_{*} + \sum
_{j=1}^{m} \abs{d_j}\norm{\mathcal{L}_{\ve}(\tilde{Z}_{0j})}_{*} \Big\}.
\end{equation}
Estimate (\ref{LpLe4a}) is a direct consequence of the following claim:

\begin{claim}
\label{Cl1}The constants $d_{j}$ are well defined,
\begin{equation}
\label{Lp27}\abs{d_j}\leq C\abs{\log\ve}\, \norm{h}_{*} \; \hbox{
and } \; \norm{\mathcal{L}_\ve (\tilde{Z}_{0j})}_{*} \leq\frac{C}%
{\abs{\log \ve}}, \; \hbox{ for all } j=1,\dots,m.
\end{equation}

\end{claim}

\medskip\noindent After these facts have been established, using the fact
that
\[
\norm{\tilde{Z}_{0j}}_{**}\leq C,
\]
we obtain (\ref{LpLe4a}), as desired.

\medskip\noindent Let us prove now Claim~\ref{Cl1}. First we find $d_{j}$.
>From definition (\ref{Lp240}), orthogonality conditions (\ref{Lp24a}) and the
fact that $\operatorname{supp} \chi_{j} \eta_{1k} =\emptyset$ and
$\operatorname{supp} \chi_{j} \eta_{2k} =\emptyset$ if $j\neq k$, we can
write
\begin{equation}
\label{App2}d_{j} \int_{\Omega_{\varepsilon}}\chi_{j} Z_{0j}^{2} = -
\int_{\Omega_{\varepsilon}} \chi_{j} Z_{0j}\psi,\quad\forall j=1,\dots,m.
\end{equation}
Thus $d_{j}$ is well defined. Note that the orthogonality conditions in
(\ref{Lp24a}) for $i=1,\dots,4$ are also satisfied for $\tilde{\psi}$ thanks
to the fact that $R>R_{0} +1$.

\medskip\noindent We prove now the second inequality in (\ref{Lp27}). From
(\ref{Lp24_1}), (\ref{Lp22a}) and estimate (\ref{AsmLe2a}) we obtain,
\begin{equation}
\label{App2a}\mathcal{L}_{\varepsilon}(\tilde{Z}_{0j})=
\begin{cases}
O\left(  \frac{\mu_{j}^{4}\varepsilon}{(\mu_{j}^{2} + r_{j}^{2})^{7/2}}
\right)  , & \hbox{
in }\Omega_{0}\\
\eta_{1j}\mathcal{L}_{\varepsilon}(Z_{0j}-\hat{Z}_{0j}) + \mathcal{L}%
_{\varepsilon}(\hat{Z}_{0j}) + F(\eta_{1j},Z_{0j}-\hat{Z}_{0j}), & \hbox{ in
} \Omega_{1}\\
\mathcal{L}_{\varepsilon}(\hat{Z}_{0j}), & \hbox{ in } \Omega_{2}\\
\eta_{2j}\mathcal{L}_{\varepsilon}(\hat{Z}_{0j}) + F(\eta_{2j},\hat{Z}%
_{0j}), & \hbox{ in }\Omega_{3},
\end{cases}
\end{equation}
and where $F$ was defined in (\ref{F1}). We compute now $\mathcal{L}%
_{\varepsilon}(\tilde{Z}_{0j})$ in $\Omega_{i}$, $i=1,2,3$. In $\Omega_{1}$,
thanks to (\ref{Lp22b}) (we consider $R$ here because we will need this
dependence below to prove estimate \eqref{App7})
\begin{equation}
\label{App4}\abs{Z_{0j}-\hat{Z}_{0j}},\; \abs{R\nabla (Z_{0j}-\hat{Z}_{0j})}\;
\hbox{ and }\; \abs{R^2\,\Delta (Z_{0j}-\hat{Z}_{0j})} = O\left(  \frac
{1}{\abs{\log\ve}}\right)  ;
\end{equation}
moreover
\begin{equation}
\label{App4a}\abs{R\,\nabla(\Delta(Z_{0j}-\hat{Z}_{0j}))}\;\hbox{
and }\; \abs{\Delta^2(Z_{0j}-\hat{Z}_{0j})} = O\left(  \frac{1}{R^{2}%
\abs{\log\ve}}\right)  ,
\end{equation}
Thus, using \eqref{F1} and the fact that, in $\Omega_{1} $, $|D^{\alpha}%
\eta_{1j} |\leq C R^{-|\alpha|}$, for any multi-index $|\alpha| \leq4$,
\[
F(\eta_{1j},Z_{0j}-\hat{Z}_{0j})=O\left(  \frac{1}{R^{4}\abs{\log\ve}}\right)
.
\]
On the other hand,
\begin{equation}
\label{3}\mathcal{L}_{\varepsilon}({Z}_{0j}-\hat{Z}_{0j})=O\left(  \frac
{1}{R^{4}\abs{\log\ve}}\right)  ,
\end{equation}
and
\begin{equation}
\label{3a}\mathcal{L}_{\varepsilon}(\hat{Z}_{0j})= O(\varepsilon R) + O\left(
\frac{1}{R^{4}\abs{\log\ve}}\right)  .
\end{equation}
In conclusion, if $y\in\Omega_{1}$,
\begin{equation}
\label{11}\displaystyle{\mathcal{L}_{\varepsilon}(\tilde{Z}_{0j})(y)=O\left(
\frac{1}{R^{4}\abs{\log\ve}}\right)  }.
\end{equation}

\bigskip\noindent In $\Omega_{2}$,
\begin{align*}
\label{111}W (1-a_{0j}G(\varepsilon y,\xi_{j}))  &  = O\left(  \frac{\mu
_{j}^{4} a_{0j}}{(\mu_{j}^{2} + r_{j}^{2})^{4}} \left\{  H(\xi_{j},\xi_{j})
-H(\varepsilon y, \xi_{j}) + 8\log\frac{r_{j}}{R} \right\}  \right) \\
&  = O\left(  \frac{\mu_{j}^{4} a_{0j}}{(\mu_{j}^{2} + r_{j}^{2})^{7/2}}
\frac{\log r_{j}}{(\mu_{j}^{2} + r_{j}^{2})^{1/2}} \right) \\
&  = O\left(  \frac{1}{\abs{\log\ve}}\frac{\mu_{j}^{4}}{(\mu_{j}^{2} +
r_{j}^{2})^{7/2}}\right)  ,
\end{align*}
and
\[
\mathcal{L}_{\varepsilon}(\hat{Z}_{0j})=O\left(  \frac{\mu_{j}^{4}
\varepsilon}{(\mu_{j}^{2} + r_{j}^{2})^{7/2}}\right)  .
\]
Thus, in this region
\begin{equation}
\label{4}\mathcal{L}(\tilde{Z}_{0j})=O\left(  \frac{\mu_{j}^{4}
\abs{\log\ve}^{-1}}{(\mu_{j}^{2} + r_{j}^{2})^{7/2}}\right)
\end{equation}

\bigskip\noindent In $\Omega_{3}$, thanks to (\ref{Lp22a}),
$\displaystyle{\ \abs{\hat{Z}_{0j}}=O\left(  \frac{1}{\abs{\log \ve}}\right)
}$, $\displaystyle{\abs{\nabla\hat{Z}_{0j}}=O\left(  \frac{\varepsilon
}{\abs{\log\ve}}\right)  }$,

$\displaystyle{\ \abs{\Delta \hat{Z}_{0j}}=O\left(  \frac{\varepsilon^{2}%
}{\abs{\log \ve}}\right)  }$,
$\displaystyle{\abs{\nabla(\Delta \hat{Z}_{0j})}=O\left(  \frac{\varepsilon
^{3}}{\abs{\log\ve}}\right)  }$ and
$\displaystyle{\abs{\Delta^2\hat{Z}_{0j}}=O\left(  \frac{\varepsilon^{4}%
}{\abs{\log\ve}}\right)  }$. Thus,
\[
F(\eta_{2j},\hat{Z}_{0j})=O\left(  \frac{\varepsilon^{4}}{\abs{\log\ve}}%
\right)  .
\]
Finally,
\begin{align*}
\mathcal{L}_{\varepsilon}(\hat{Z}_{0j})  &  = \mathcal{L}_{\varepsilon}%
(Z_{0j}) + W a_{0j}\left(  H(\xi_{j},\xi_{j}) - H(\varepsilon y,\xi_{j}) +
8\log\frac{r_{j}}{R} \right) \\
&  = O\left(  \frac{\mu_{j}^{4} \varepsilon}{(\mu_{j}^{2} + r_{j}^{2})^{7/2}}
\right)  + O\left(  \frac{\mu_{j}^{4}}{(\mu_{j}^{2} + r_{j}^{2})^{4}} \right)
\\
&  = O\left(  \frac{\mu_{j}^{4} \varepsilon}{(\mu_{j}^{2} + r_{j}^{2})^{7/2}}
\right)
\end{align*}
and then, combining (\ref{11}), (\ref{4}) and the previous estimate, we can
again write the estimate (\ref{App2a}):
\begin{equation}
\label{App2b}\mathcal{L}_{\varepsilon}(\tilde{Z}_{0j})=
\begin{cases}
O\left(  \frac{\mu_{j}^{4}\varepsilon}{(\mu_{j}^{2} + r_{j}^{2})^{7/2}}
\right)  , & \hbox{
in }\Omega_{0}\\
O\left(  \frac{1}{\abs{\log\ve}} \right)  , & \hbox{ in }\Omega_{1}\\
O\left(  \frac{\mu_{j}^{4}\abs{\log \ve}^{-1}}{(\mu_{j}^{2} + r_{j}^{2}%
)^{7/2}} \right)  , & \hbox{ in }\Omega_{2}\\
O\left(  \frac{\mu_{j}^{4}\varepsilon}{(\mu_{j}^{2} + r_{j}^{2})^{7/2}%
}\right)  , & \hbox{ in } \Omega_{3}.
\end{cases}
\end{equation}
In conclusion,
\begin{equation}
\label{norm1}\norm{\mathcal{L}_{\ve}(\tilde{Z}_{0j})}_{*}=O\left(  \frac
{1}{\abs{\log \ve}}\right)  .
\end{equation}
Finally, we prove the bounds of $d_{j}$. Testing equation (\ref{Lp25}) against
$\tilde{Z}_{0j}$ and using relations (\ref{Lp26}) and the above estimate, we
get
\begin{align*}
\abs{d_j} \left|  \int_{\Omega_{\varepsilon}}\mathcal{L}_{\varepsilon}%
(\tilde{Z}_{0j})\tilde{Z}_{0j} \right|   &  = \left|  \int_{\Omega
_{\varepsilon}}h\tilde{Z}_{0j} + \int_{\Omega_{\varepsilon}}\tilde{\psi
}\mathcal{L}_{\varepsilon}(\tilde{Z}_{0j}) \right| \\
&  \leq C\norm{h}_{*} + C\norm{\tilde{\psi}}_{\infty}%
\norm{\mathcal{L}_{\ve}(\tilde{Z}_{0j})}_{*}\\
&  \leq C\norm{h}_{*} \left\{  1 +
\norm{\mathcal{L}_{\ve}(\tilde{Z}_{0j})}_{*} \right\}  + C\sum_{k=1}^{m}
\abs{d_k}\norm{\mathcal{L}_{\ve}(\tilde{Z}_{0k})}_{*}%
\norm{\mathcal{L}_{\ve}(\tilde{Z}_{0j})}_{*}%
\end{align*}
where we have used that
\[
\displaystyle{\int_{\Omega_{\varepsilon}}\frac{\mu_{j}^{4}}{(\mu_{j}^{2} +
r_{j}^{2})^{7/2}} \leq C} \quad\hbox{for all } j .
\]
But estimate (\ref{norm1}) imply
\begin{align}
\label{App6}\abs{d_j}\left|  \int_{\Omega_{\varepsilon}}\mathcal{L}%
_{\varepsilon}(\tilde{Z}_{0j})\tilde{Z}_{0j} \right|  \leq C\norm{h}_{*} +
C\sum_{k=1}^{m} \frac{\abs{d_k}}{\abs{\log \ve}^{2}}.
\end{align}
It only remains to estimate the integral term of the left side. For this
purpose, we have the following

\begin{claim}
\label{Cl3} If $R$ is sufficiently large,
\begin{equation}
\label{App7}\left|  \int_{\Omega_{\varepsilon}} \mathcal{L}_{\varepsilon
}(\tilde{Z}_{0j})\tilde{Z}_{0j} \right|  = \frac{E}{\abs{\log\ve}}(1+o(1)),
\end{equation}
where $E$ is a positive constant independent of $\varepsilon$ and $R$.
\end{claim}

\medskip\noindent Assume for the moment the validity of this claim. We replace
(\ref{App7}) in (\ref{App6}), we get
\begin{equation}
\label{App9}\abs{d_j}\leq C\abs{\log\ve}\norm{h}_{*} + C\sum_{k=1}^{m}
\frac{\abs{d_k}}{\abs{\log\ve}},
\end{equation}
and then,
\[
\abs{d_j}\leq C\abs{\log\ve}\norm{h}_{*}.
\]
Claim \ref{Cl1} is thus proven. Let us proof Claim \ref{Cl3}. We decompose
\begin{align*}
\int_{\Omega_{\varepsilon}}\mathcal{L}_{\varepsilon}(\tilde{Z}_{0j})\tilde
{Z}_{0j}  &  = O(\varepsilon) + \int_{\Omega_{1}}\mathcal{L}_{\varepsilon
}(\tilde{Z}_{0j})\tilde{Z}_{0j} + \int_{\Omega_{2}}\mathcal{L}_{\varepsilon
}(\tilde{Z}_{0j})\tilde{Z}_{0j} + \int_{\Omega_{3}}\mathcal{L}_{\varepsilon
}(\tilde{Z}_{0j})\tilde{Z}_{0j}\\
&  \equiv O(\varepsilon) + I_{1} + I_{2} + I_{3}.
\end{align*}
First we estimate $I_{2}$. From (\ref{App2b}),
\begin{align*}
I_{2}  &  = O\left(  \frac{1}{\abs{\log\ve}}\int_{\Omega_{2}}\frac{\mu_{j}^{4}
\hat{Z}_{0j}}{(\mu_{j}^{2} + r_{j}^{2})^{7/2}}\right) \\
&  = O\left(  \frac{1}{R^{3}\abs{\log\ve}}\right)  .
\end{align*}
Now we estimate $I_{3}$. From the estimates in $\Omega_{3}$,
$\displaystyle{\ \abs{I_3}=O\left(  \frac{\varepsilon^{4}}{\abs{\log\ve}}%
\right)  }$. On the other hand, since \eqref{11} holds true and $\hat{Z}%
_{0j}=Z_{0j}\left(  1+ O\left(  \frac{1}{R\abs{\log\ve}}\right)  \right)  $,
we conclude
\begin{align*}
\abs{I_1}  &  = \frac{1}{R^{4}\abs{\log\ve}}\int_{R<r_{j}\leq R+1}\tilde
{Z}_{0j}(y)\,dy\\
&  = \frac{1}{R^{4}\abs{\log\ve}}\int_{R<r_{j}\leq R+1} \left\{  O\left(
\frac{1}{R\abs{\log\ve}} \right)  + \hat{Z}_{0j}(y) \right\}  \,dy\\
&  = \frac{1}{R^{5}\abs{\log\ve}^{2}} + \frac{|S^{3}|}{R^{4}\abs{\log\ve}}
\int_{R}^{R+1} r^{3} \left(  \frac{r^{2}-\mu_{j}^{2}}{\mu_{j}^{2}+ r^{2}%
}\right)  (1+o(1))\,dr\\
&  = \frac{E}{\abs{\log\ve}}(1+o(1)),
\end{align*}
where $E$ is a positive constant independent of $\varepsilon$ and $R$. Thus,
for fixed $R$ large and $\varepsilon$ small, we obtain (\ref{App7}).
\end{proof}

\medskip\noindent Now we can try with the original linear problem (\ref{Lp5}).

\medskip\noindent\textbf{Proof of Proposition~\ref{LpTeo1}}. We first
establish the validity of the a priori estimate (\ref{LpTeo1a}) for solutions
$\psi$ of problem (\ref{Lp5}), with $h\in L^{\infty}(\Omega_{\varepsilon})$
and $\| h \|_{*} <\infty$. Lemma \ref{LpLe4} implies
\begin{equation}
\label{Lp40}\norm{\psi}_{**} \leq C\abs{\log \ve} \Big\{ \norm{h}_{*} +
\sum_{i=1}^{2} \sum_{j=1}^{m} \abs{c_{ij}}\norm{\chi_jZ_{ij}}_{*}\Big\}.
\end{equation}
On the other hand,
\[
\norm{\chi_jZ_{ij}}_{*}\leq C,
\]
then, it is sufficient to estimate the values of the constants $c_{ij}$. To
this end, we multiply the first equation in (\ref{Lp5}) by $Z_{ij}\eta_{2j}$,
with $\eta_{2j}$ the cut-off function introduced in (\ref{Lp23}), and
integrate by parts to find
\begin{equation}
\label{Lp41}\int_{\Omega_{\varepsilon}}\psi\mathcal{L}_{\varepsilon}%
(Z_{ij}\eta_{2j})=\int_{\Omega_{\varepsilon}}hZ_{ij}\eta_{2j} + c_{ij}%
\int_{\Omega_{\varepsilon}}\eta_{2j}Z_{ij}^{2},
\end{equation}
It is easy to see that $\displaystyle{\int_{\Omega_{\varepsilon}}\eta
_{2j}Z_{ij}h=O(\norm{h}_{*})}$ and $\displaystyle{\int_{\Omega_{\varepsilon}%
}\eta_{2j}Z_{ij}^{2}=C>0}$. On the other hand we have
\begin{align*}
\label{Lp42}\mathcal{L}_{\varepsilon}(\eta_{2j}Z_{ij})  &  = \eta
_{2j}\mathcal{L}_{\varepsilon}(Z_{ij}) + F(\eta_{2j},Z_{ij})\\
&  = O\left(  \frac{\mu_{j}^{4}\varepsilon}{(\mu_{j}^{2} + r_{j}^{2})^{7/2}%
}\right)  \eta_{2j}\abs{Z_{ij}} + F(\eta_{2j},Z_{ij}).
\end{align*}
Directly from \eqref{F1} we get
\begin{align*}
F(\eta_{2j},Z_{ij})  &  = O\left(  \frac{\varepsilon^{4}}{(\mu_{j}^{2} +
r_{j}^{2})^{1/2}}\right)  + O\left(  \frac{\varepsilon^{3}}{\mu_{j}^{2} +
r_{j}^{2}}\right) \\
&  + O\left(  \frac{\varepsilon^{2}}{(\mu_{j}^{2} + r_{j}^{2})^{3/2}}\right)
+ O\left(  \frac{\varepsilon}{(\mu_{j}^{2} + r_{j}^{2})^{2}}\right)  ,
\end{align*}
in the region $\frac{\delta_{0}}{3\varepsilon} \leq r_{j}\leq\frac{\delta_{0}%
}{2\varepsilon}$. Thus
\begin{equation}
\label{N1}\norm{\mathcal{L}_{\ve}(\eta_{2j}Z_{ij})}_{*}= O(\varepsilon)
\quad\hbox{ and }
\end{equation}
\[
\left|  \int_{\Omega_{\varepsilon}} \psi\mathcal{L}_{\varepsilon}(\eta
_{2j}Z_{ij})\right|  \leq C\varepsilon\abs{\log\ve}\norm{\psi}_{\infty} \leq
C\varepsilon\abs{\log\ve}\norm{\psi}_{**}.
\]
Using the above estimates in (\ref{Lp41}), we obtain
\begin{equation}
\label{Lp43}\abs{c_{ij}}\leq C\{\varepsilon\abs{\log\ve}\norm{\psi}_{**} +
\norm{h}_{*} \},
\end{equation}
and then
\[
\displaystyle{\abs{c_{ij}}\leq C \Big\{ (1 + \varepsilon\abs{\log\ve}^{2})
\norm{h}_{*} + \varepsilon\abs{\log\ve}^{2}\sum_{l,k}\abs{c_{lk}} \Big\}}.
\]
Then $\abs{c_{ij}}\leq C \norm{h}_{*}$ and putting this estimate in
(\ref{Lp40}), we conclude the validity of (\ref{LpLe4a}).

\medskip\noindent We now prove the solvability assertion. To this purpose we
consider the space
\begin{align*}
\mathcal{H}  &  = \Big\{ \psi\in H^{3}(\Omega_{\varepsilon})\;:\; \psi
=\Delta\psi=0 \hbox{ on } \partial\Omega_{\varepsilon},
\hbox{ and such that }\\
&  \quad\quad\int_{\Omega_{\varepsilon}} \chi_{j}Z_{ij}\psi=0,\quad
\hbox{ for all } i=1,\dots,4;\, j=1,\dots,m \Big\},
\end{align*}
endowed with the usual inner product $(\psi,\varphi)=\int_{\Omega
_{\varepsilon}}\Delta\psi\Delta\varphi$. Problem (\ref{Lp21}) expressed in a
weak form is equivalent to that of finding a $\psi\in\mathcal{H}$, such that
\[
(\psi,\varphi) =\int_{\Omega_{s}} \big\{h + W\psi\big\}\varphi, \quad
\hbox{ for all } \varphi\in\mathcal{H}.
\]
With the aid of Riesz's representation Theorem, this equation can be rewritten
in $\mathcal{H}$ in the operator form $\psi=K(W\psi+ h)$, where $K $ is a
compact operator in $\mathcal{H}$. Fredholm's alternative guarantees unique
solvability of this problem for any $h$ provided that the homogeneous equation
$\psi=K(W\psi)$ has only the zero solution in $\mathcal{H}$. This last
equation is equivalent to (\ref{Lp21}) with $h\equiv0$. Thus existence of a
unique solution follows from the a priori estimate (\ref{LpLe4a}). This
concludes the proof. \qed

\bigskip\noindent The result of Proposition~\ref{LpTeo1} implies that the
unique solution $\psi=T(h)$ of (\ref{Lp5}) defines a continuous linear map
from the Banach space $\mathcal{C}_{*}$ of all functions $h\in L^{\infty
}(\Omega_{\varepsilon})$ with $\norm{h}_{*} < +\infty$, into $W^{3,\infty
}(\Omega_{\varepsilon})$, with norm bounded uniformly in $\varepsilon$.

\begin{remark}
\label{Re1} The operator $T$ is differentiable with respect to the variables
$\xi^{\prime}$. In fact, computations similar to those used in \cite{KMP}
yield the estimate
\begin{equation}
\label{Re1a}\norm{\partial_{\xi'}T(h)}_{**}\leq C\abs{\log\ve}^{2}
\norm{h}_{*},\quad\hbox{ for all }l=1,2;\;k=1,\dots,m.
\end{equation}

\end{remark}

\bigskip

\section{The intermediate nonlinear problem}

\noindent In order to solve Problem (\ref{Asm7}) we consider first the
intermediate nonlinear problem.
\begin{align}
\label{NLP1}%
\begin{cases}
& \mathcal{L}_{\varepsilon}(\psi)= -R + N(\psi) + \sum_{i=1}^{4} \sum
_{j=1}^{m} c_{ij}\chi_{j}Z_{ij}, \quad\hbox{ in } \Omega_{\varepsilon},\\
& \psi=\Delta\psi=0, \quad\hbox{ on } \partial\Omega_{\varepsilon},\\
& \int_{\Omega_{\varepsilon}}\chi_{j}Z_{ij}\psi=0, \hbox{ for all }
i=1,\dots,4\; j=1,\dots, m.
\end{cases}
\end{align}
For this problem we will prove

\begin{prop}
\label{NLPLe1} Let $\xi\in\mathcal{O}$. Then, there exists $\varepsilon_{0}>0$
and $C>0$ such that for all $\varepsilon\leq\varepsilon_{0}$ the nonlinear
problem (\ref{NLP1}) has a unique solution $\psi\in$ which satisfies
\begin{equation}
\label{NLP2}\norm{\psi}_{**} \leq C\,\varepsilon\abs{\log\ve}.
\end{equation}
Moreover, if we consider the map $\xi^{\prime}\in\mathcal{O}\to\psi
\in\mathcal{C}^{4,\alpha}(\overline{\Omega}_{\varepsilon})$, the derivative
$D_{\xi^{\prime}}\psi$ exists and defines a continuous map of $\xi^{\prime}$.
Besides
\begin{equation}
\label{NLP3}\norm{D_{\xi'}\psi}_{**} \leq C\,\varepsilon\abs{\log\ve}^{2}.
\end{equation}

\end{prop}

\begin{proof}
In terms of the operator $T$ defined in Proposition \ref{LpTeo1}, Problem
(\ref{NLP1}) becomes
\[
\psi= \mathcal{B}(\psi)\equiv T(N(\psi) - R).
\]
Let us consider the region
\[
\mathcal{F}\equiv\{ \psi\in\mathcal{C}^{4,\alpha}(\overline{\Omega
}_{\varepsilon}) \, : \, \norm{\psi}_{**} \leq\varepsilon\abs{\log\ve} \}.
\]
>From Proposition~\ref{LpTeo1},
\[
\norm{\mathcal{B}(\psi)}_{**}\leq C\,\abs{\log\ve}\,\big\{
\norm{N(\psi)}_{*} + \norm{R}_{*} \big\},
\]
and Lemma~\ref{AsmLe1} implies
\[
\norm{R}_{*}\leq C\varepsilon.
\]
Also, from Lemma~\ref{AsmLe2}
\[
\norm{N(\psi)}_{*} \leq C\norm{\psi}^{2}_{\infty}\leq C\norm{\psi}_{**}^{2}.
\]
Hence, if $\psi\in\mathcal{F}$, $\norm{\mathcal{B}(\psi)}_{**}\leq
C\varepsilon\abs{\log\ve}$. Along the same way we obtain
$\norm{N(\psi_1)-N(\psi_2)}_{*}\leq C \max_{i=1,2}\norm{\psi_i}_{\infty
}\norm{\psi_1-\psi_2}_{\infty}\leq C\max_{i=1,2}\norm{\psi_i}_{**}%
\norm{\psi_1-\psi_2}_{**}$, for any $\psi_{1},\psi_{2}\in\mathcal{F}$. Then,
we conclude
\[
\norm{\mathcal{B}(\psi_1)-\mathcal{B}(\psi_2)}_{**}\leq
C\abs{\log\ve}\,\norm{N(\psi_1)-N(\psi_2)}_{*}\leq C\varepsilon
\abs{\log\ve}^{2}\norm{\psi_1-\psi_2}_{**}%
\]
It follows that for all $\varepsilon$ small enough $\mathcal{B}$ is a
contraction mapping of $\mathcal{F}$, and therefore a unique fixed point of
$\mathcal{B}$ exists in this region. The proof of (\ref{NLP3}) is similar to
one included in \cite{KMP} and we thus omit it.
\end{proof}

\medskip


\section{Variational reduction}

\noindent We have solved the nonlinear problem (\ref{NLP1}). In order to find
a solution to the original problem (\ref{Asm7}) we need to find $\mathbf{\xi}$
such that
\begin{equation}
\label{VR1}c_{ij}=c_{ij}(\mathbf{\xi^{\prime}})=0,\quad\hbox{for all
} i,j.
\end{equation}
where $c_{ij}(\mathbf{\xi^{\prime}})$ are the constants in (\ref{NLP1}).
Problem (\ref{VR1}) is indeed variational: it is equivalent to finding
critical points of a function of $\mathbf{\xi^{\prime}}$. In fact, we define
the function for $\mathbf{\xi}\in\mathcal{O}$
\begin{equation}
\label{VR2}\mathcal{F}_{\varepsilon}(\mathbf{\xi}) \equiv J_{\rho
}[U(\mathbf{\xi})+\hat{\psi}_{\mathbf{\xi}}]
\end{equation}
where $J_{\rho}$ is defined in \eqref{In4}, $\rho$ is given by \eqref{rhoeps},
$U=U(\mathbf{\xi})$ is our approximate solution from (\ref{Aa5a}) and
$\hat{\psi}_{\mathbf{\xi}}=\psi\big(\frac{x}{\varepsilon},\frac{\mathbf{\xi}%
}{\varepsilon}\big)$, $x\in\Omega$, with $\psi=\psi_{\xi^{\prime}}$ the unique
solution to problem (\ref{NLP1}) given by Proposition \ref{NLPLe1}. Then we
obtain that critical points of $\mathcal{F}$ correspond to solutions of
(\ref{VR1}) for small $\varepsilon$. That is,

\begin{lema}
\label{VRLe1} $\mathcal{F}_{\varepsilon}:\mathcal{O}\to\mathbb{R}$ is of class
$\mathcal{C}^{1}$. Moreover, for all $\varepsilon$ small enough, if
$D_{\mathbf{\xi}}\mathcal{F}_{\varepsilon}(\mathbf{\xi})=0$ then $\mathbf{\xi
}$ satisfies (\ref{VR1}).
\end{lema}

\begin{proof}
We define
\[
I_{\varepsilon}[v]\equiv\frac12 \int_{\Omega_{\varepsilon}}(\Delta v)^{2} -
\int_{\Omega_{\varepsilon}}k(\varepsilon y)e^{v}.
\]
Let us differentiate the function $\mathcal{F}_{\varepsilon}$ with respect to
$\xi$. Since $J_{\rho}[U(\xi) +\hat{\psi}_{\xi}]=I_{\varepsilon}[V(\xi
^{\prime})+\psi_{\xi^{\prime}}]$, we can differentiate directly under the
integral sign, so that
\begin{align*}
\label{Sola1}\partial_{(\xi_{k})_{l}}\mathcal{F}_{\varepsilon}(\xi)  &  =
\varepsilon^{-1}DI_{\varepsilon} [V+\psi]\left(  \partial_{(\xi_{k}^{\prime
})_{l}}V + \partial_{(\xi_{k}^{\prime})_{l}}\psi\right) \\
&  = \varepsilon^{-1}\sum_{i=1}^{4}\sum_{j=1}^{m} \int_{\Omega_{\varepsilon}%
}c_{ij}\chi_{j}Z_{ij}\left(  \partial_{(\xi_{k}^{\prime})_{l}}V +
\partial_{(\xi_{k}^{\prime})_{l}}\psi\right)  .
\end{align*}
From the results of the previous section, this expression defines a continuous
function of $\xi^{\prime}$, and hence of $\xi$. Let us assume that $D_{\xi
}\mathcal{F}_{\varepsilon}(\xi)=0$. Then
\[
\sum_{i=1}^{4}\sum_{j=1}^{m} \int_{\Omega_{\varepsilon}}c_{ij}\chi_{j}%
Z_{ij}\left(  \partial_{(\xi_{k}^{\prime})_{l}}V + \partial_{(\xi_{k}^{\prime
})_{l}}\psi\right)  =0,\quad\hbox{ for } k=1,2,3,4;\quad l=1,\dots,m.
\]
Since $\norm{D_{\xi'}\psi_{\xi'}}\leq C\varepsilon\abs{\log\ve}^{2}$, we have
\[
\partial_{(\xi_{k}^{\prime})_{l}}V + \partial_{(\xi_{k}^{\prime})_{l}}\psi=
Z_{kl} + o(1),
\]
where $o(1)$ is uniformly small as $\varepsilon\to0$. Thus, we have the
following linear system of equation
\[
\sum_{i=1}^{4}\sum_{j=1}^{m} c_{ij}\int_{\Omega_{\varepsilon}} \chi_{j}Z_{ij}(
Z_{kl} + o(1))=0,\quad\hbox{ for } k=1,2,3,4;\quad l=1,\dots,m.
\]
This system is dominant diagonal, thus $c_{ij}=0$ for all $i,j$. This
concludes the proof.
\end{proof}

\noindent We also have the validity of the following Lemma

\begin{lema}
\label{VRLe2} Let $\rho$ be given by \eqref{rhoeps}. For points $\mathcal{\xi
}\in\mathcal{O}$ the following expansion holds
\begin{equation}
\label{VRLe21}\mathcal{F}_{\varepsilon}(\mathbf{\xi})= J_{\rho}[U(\mathbf{\xi
})] +\theta_{\varepsilon}(\mathbf{\xi}),
\end{equation}
where $\abs{\theta_{\ve}} + \abs{\nabla \theta_{\ve}}=o(1)$, uniformly on
$\xi\in\mathcal{O}$ as $\varepsilon\to0$.
\end{lema}

\begin{proof}
The proof follows directly from an application of Taylor expansion for
$\mathcal{F}_{\varepsilon}$ in the expanded domain $\Omega_{\varepsilon}$ and
from the estimates for the solution $\psi_{\xi^{\prime}}$ to Problem
(\ref{NLP1}) obtained in Proposition \ref{NLPLe1}.
\end{proof}

\section{Proof of the theorems}

\noindent In this section we carry out the proofs of our main results.

\subsection{Proof of Theorem \ref{teo1}}

\noindent Taking into account the result of Lemma \ref{VRLe1}, a solution to
Problem \eqref{In1} exists if we prove the existence of a critical point of
$\mathcal{F}_{\varepsilon}$, which automatically implies that $c_{ij}=0$ in
\eqref{Asm7} for all $i,j$. The qualitative properties of the solution found
follow from the ansatz.

\noindent Finding critical points of $\mathcal{F}_{\varepsilon}(\xi)$ is
equivalent to finding critical points of
\begin{equation}
\tilde{\mathcal{F}}_{\varepsilon}(\xi)=\mathcal{F}_{\varepsilon}(\xi
)-256\,\pi^{2}\,m\abs{\log\ve}. \label{uno}%
\end{equation}
On the other hand, if $\xi\in\mathcal{O}$, from Lemmas \ref{Ene1} and
\ref{VRLe2} we get the existence of constants $\alpha>0$ and $\beta$ such
that
\begin{equation}
\alpha\tilde{\mathcal{F}}_{\varepsilon}(\xi)+\beta=\varphi_{m}(\xi
)+\varepsilon\Theta_{\varepsilon}(\xi), \label{due}%
\end{equation}
with $\Theta_{\varepsilon}$ and $\nabla_{\xi}\Theta_{\varepsilon}$ uniformly
bounded in the considered region as $\varepsilon\rightarrow0$.

\noindent We shall prove that, under the assumptions of Theorems \ref{teo1}
and \ref{teo2}, $\tilde{\mathcal{F}}_{\varepsilon}$ has a critical point in
${\mathcal{O}}$ for $\varepsilon$ small enough. We start with a topological
lemma. We denote by $D$ the diagonal
\[
D:=\{\xi\in\Omega^{m}\,:\,\xi_{i}=\xi_{j}{\mbox { for
some }}i\not =j\},
\]
and we write $H^{\ast}:=H^{\ast}($ $\cdot$ $;\mathbb{K})$ for singular
cohomology with coefficients in a field $\mathbb{K}$.

\begin{lema}
\label{monica} If $H^{d}(\Omega)\not =0$ for some $d\geq1$, and $H^{j}%
(\Omega)=0$ for $j>d,$ then the homomorphism
\[
H^{md}(\Omega^{m},D)\longrightarrow H^{md}(\Omega^{m}),
\]
induced by the inclusion of pairs $(\Omega^{m},\emptyset)\hookrightarrow
(\Omega^{m},D),$ is an epimorphism. In particular, $H^{md}(\Omega^{m}%
,D)\not =0$.
\end{lema}

\begin{proof}
Let us prove first that $H^{j}(D)=0$ if $j>(m-1)d$. For this purpose we write
\[
D=\bigcup_{1\leq i<j\leq m}X_{i,j},\quad{\mbox {where}}\quad X_{i,j}%
:=\{(x_{1},\ldots,x_{m})\in\Omega^{m}\,:\,x_{i}=x_{j}\},
\]
and consider the sets $\mathcal{F}_{0}:=\{\Omega^{m}\}$, $\mathcal{F}%
_{1}:=\{X_{i,j}:1\leq i<j\leq m\}$, and
\[
\mathcal{F}_{k+1}:=\{Z\cap Z^{\prime}:Z,Z^{\prime}\in\mathcal{F}_{k}\text{ and
}Z\neq Z^{\prime}\},\text{ \ \ }k=1,...,m-2.
\]
Note that
\[
Z\cong\Omega^{m-k^{\prime}}\text{ for some }k\leq k^{\prime}\leq m-1\text{
\ if }Z\in\mathcal{F}_{k},\text{ }k=0,...,m-1,
\]
where $\cong$ means that the sets are homeomorphic. K\"{u}nneth's formula%
\begin{equation}
H^{j}(\Omega^{m-k})=%
{\textstyle\bigoplus\limits_{p+q=j}}
\left(  H^{p}(\Omega)\otimes H^{q}(\Omega^{m-k-1})\right)  \label{kf}%
\end{equation}
(see, for example, \cite[Proposition 8.18]{d}) yields inductively that, for
$0\leq k\leq m-1$,%
\begin{equation}
H^{j}(Z)=0\quad{\mbox{if}}\quad Z\in\mathcal{F}_{k}\quad{\mbox{and}}\quad
j>(m-k)d.\label{previous}%
\end{equation}
We claim that, for each $0\leq k\leq m-1,$ one has that%
\begin{equation}
H^{j}(Z_{1}\cup\cdot\cdot\cdot\cup Z_{\ell})=0\ \ \ \text{if }Z_{1}%
,...,Z_{\ell}\in\mathcal{F}_{k}\text{ and }j>(m-k)d.\label{claim}%
\end{equation}
Let us prove this claim. Since $\mathcal{F}_{m-1}$ has only one element and
(\ref{previous}) holds, we have that the claim is true for $k=m-1.$ Assume
that the claim is true for $k+1$ with $k+1\leq m-1$ and let us then prove it
for $k.$ We do this by induction on $\ell.$ If $\ell=1$ the assertion reduces
to (\ref{previous}). Now assume that the assertion is true for every union of
at most $\ell-1$ sets in $\mathcal{F}_{k},$ and let $Z_{1},...,Z_{\ell}%
\in\mathcal{F}_{k}$ be pairwise distinct sets. Consider the Mayer-Vietoris
sequence%
\begin{equation}
\cdot\cdot\cdot\rightarrow H^{j-1}\left(
{\textstyle\bigcup\limits_{i=1}^{\ell-1}}
(Z_{i}\cap Z_{\ell})\right)  \rightarrow H^{j}(Z_{1}\cup\cdot\cdot\cdot\cup
Z_{\ell})\rightarrow H^{j}(Z_{1}\cup\cdot\cdot\cdot\cup Z_{\ell-1})\oplus
H^{j}(Z_{\ell})\rightarrow\cdot\cdot\cdot\label{mv}%
\end{equation}
Our induction hypothesis on $\ell$ yields that $H^{j}(Z_{1}\cup\cdot\cdot
\cdot\cup Z_{\ell-1})=0$ and $H^{j}(Z_{\ell})=0$ if $j>(m-k)d$. Since
$Z_{1},...,Z_{\ell}$ are pairwise distinct, we have that $Z_{i}\cap Z_{\ell
}\in\mathcal{F}_{k+1}$ for each $i=1,...,\ell-1$ and, since we are assuming
that the claim is true for $k+1$ we have that
\[
H^{j-1}\left(
{\textstyle\bigcup\limits_{i=1}^{\ell-1}}
(Z_{i}\cap Z_{\ell})\right)  =0\text{ \ if }j-1>(m-(k+1))d.
\]
Note that $j>(m-k)d$ implies $j-1>(m-(k+1))d.$ This proves that both ends of
the exact sequence (\ref{mv}) are zero if $j>(m-k)d,$ hence the middle term is
also zero in this case. This concludes the proof of claim (\ref{claim}).

\noindent Now, since $D=\bigcup_{Y\in\mathcal{F}_{1}}Y,$ assertion
(\ref{claim}) with $k=1$ yields that $H^{j}(D)=0$ if $j>(m-1)d.$ So the exact
cohomology sequence
\[
H^{md}(\Omega^{m},D)\longrightarrow H^{md}(\Omega^{m})\longrightarrow
H^{md}(D)=0
\]
gives that $H^{md}(\Omega^{m},D)\longrightarrow H^{md}(\Omega^{m})$ is an
epimorphism. But (\ref{kf}) implies that $H^{md}(\Omega^{m})\neq0.$ Therefore,
$H^{md}(\Omega^{m},D)\neq0,$ as claimed.
\end{proof}

For each positive number $\delta$ define
\begin{align*}
\Omega_{\delta}  &  :=\{\xi\in\Omega\,:\,{\mbox {dist}}(\xi,\partial
\Omega)>\delta\}\\
{\mathfrak{D}}_{\delta}  &  :=\{\xi=(\xi_{1},\ldots,\xi_{m})\in\Omega^{m}%
:\xi_{j}\in\Omega_{\delta}\}.
\end{align*}

\begin{lema}
\label{bdryflow}Given $K>0$ there exists $\delta_{0}>0$ such that, for each
$\delta\in(0,\delta_{0}),$ the following holds: For every $\xi=(\xi
_{1},...,\xi_{m})\in\partial${$\mathfrak{D}$}$_{\delta}$ with $\left\vert
\varphi_{m}(\xi)\right\vert \leq K$ there exists an $i\in\{1,...,m\}$ such
that%
\[%
\begin{array}
[c]{ll}%
\nabla_{\xi_{i}}\varphi_{m}(\xi)\not =0 & \text{if }\xi_{i}\in\Omega_{\delta
}\\
\nabla_{\xi_{i}}\varphi_{m}(\xi)\cdot\tau\not =0\text{ \ for some }\tau\in
T_{\xi_{i}}(\partial\Omega_{\delta}) & \text{if }\xi_{i}\in\partial
\Omega_{\delta}%
\end{array}
\]
where $T_{\xi_{i}}(\partial\Omega_{\delta})$ denotes the tangent space to
$\partial\Omega_{\delta}$ at the point $\xi_{i}.$
\end{lema}

\noindent\textit{Proof}. \ We first need to establish some facts related to
the regular part of the Green function on the half hyperplane
\[
\mathcal{H}:=\left\{  x=(x_{1},x_{2},x_{3},x_{4})\in\mathbb{R}^{4}%
\;:\;x_{4}\geq0\right\}  .
\]
It is well known that the regular part of the Green function on $\mathcal{H}$
is given by
\[
H(x,y)=8\log\abs{x-\bar{y}},\quad\bar{y}=(y_{1},y_{2},y_{3},-y_{4}),
\]
for $x,y\in\mathcal{H}$ and the Green function is
\[
\displaystyle{G(x,y)=-8\log\abs{x-y}+8\log\abs{x-\bar{y}}}.
\]
Consider the function of $k\geq2$ distinct points of $\mathcal{H}$
\[
\Psi_{k}(x_{1},\dots,x_{k}):=-8\sum_{i\neq j}\log\abs{x_i-x_j},
\]
and denote by $I_{+}$ and $I_{0}$ the set of indices $i$ for which
$(x_{i})_{4}>0$ and $(x_{i})_{4}=0$, respectively. Define also
\[
\varphi_{k,\mathcal{H}}(x_{1},\dots,x_{k})=-8\sum_{j=1}^{k}\log
\abs{x_j-\bar{x_j}}+8\sum_{i\neq j}\log\frac{\abs{x_i-x_j}}%
{\abs{x_i-\bar{x_j}}}.
\]

\begin{claim}
\label{Clf1} We have the following alternative: Either
\[
\nabla_{x_{i}}\Psi_{k}(x_{1},\dots,x_{k})\neq0\quad\hbox{ for some }i\in
I_{+},
\]
or
\[
\partial_{(x_{i})_{j}}\Psi_{k}(x_{1},\dots,x_{k})\neq0\quad\hbox{ for
some }i\in I_{0}\hbox{ and }j\in\{1,2,3\},
\]
where $\partial_{(x_{i})_{j}} \equiv\frac{\partial}{\partial(x_{i})_{j}}$.
\end{claim}

\noindent\textit{Proof}. \ We have that
\[
\frac{\partial}{\partial\lambda}\Psi_{k}(\lambda x_{1},\dots,\lambda
x_{k})\big|_{\lambda=1}=\sum_{i\in I_{+}}\nabla_{x_{i}}\Psi_{k}(x_{1}%
,\dots,x_{k})\cdot x_{i}+\sum_{i\in I_{0}}\nabla_{x_{i}}\Psi_{k}(x_{1}%
,\dots,x_{k})\cdot x_{i}.
\]
On the other hand
\[
\frac{\partial}{\partial\lambda}\Psi_{k}(\lambda x_{1},\dots,\lambda
x_{k})\big|_{\lambda=1}=-8k(k-1)\neq0,
\]
and Claim \ref{Clf1} follows.

\begin{claim}
\label{Clf2} For any $k$ distinct points $x_{i}\in\operatorname{Int}%
\mathcal{H}$ we have $\displaystyle{\nabla\varphi_{k,\mathcal{H}}}(x_{1}%
,\dots,x_{k})\neq0$.
\end{claim}

\noindent\textit{Proof}. \ We have that
\[
\frac{\partial}{\partial\lambda}\varphi_{k,\mathcal{H}}(\lambda x_{1}%
,\dots,\lambda x_{k})\big|_{\lambda=1}=\sum_{i=1}^{k}\nabla_{x_{i}}%
\varphi_{k,\mathcal{H}}(x_{1},\dots,x_{k})\cdot x_{i},
\]
On the other hand
\[
\frac{\partial}{\partial\lambda}\varphi_{k,\mathcal{H}}(\lambda x_{1}%
,\dots,\lambda x_{k})\big|_{\lambda=1}=-8k(k-1)\neq0,
\]
and Claim \ref{Clf2} follows.

\noindent Now we will need an estimate for the regular part $H(x,y)$ of the
Green's function for points $x,y$ close to $\partial\Omega$.

\begin{claim}
\label{Clf3} There exists $C_{1},C_{2}>0$ constants such that for any $x,
y\in\Omega$
\[
\abs{\nabla _x H(x,y)} + \abs{\nabla _y H(x,y)} \leq C_{1} \min\left\{
\frac{1}{\abs{x-y}}, \frac{1}{\operatorname{dist} (y,\partial\Omega)}\right\}
+ C_{2}.
\]

\end{claim}

\noindent\textit{Proof}. \ For $y\in\Omega$ a point close to $\partial\Omega$
we denote by $\bar{y}$ its uniquely determined reflection with respect to
$\partial\Omega$. Define $\psi(x,y)=H(x,y)+8\log{\frac{1}{|x-\bar{y}|}}$. It
is straightforward to see that $\psi$ is bounded in $\bar{\Omega}\times
\bar{\Omega}$ and that $|\nabla_{x}\psi(x,y)|+|\nabla_{y}\psi(x,y)|\leq C$ for
some positive constant $C$. Claim \ref{Clf3} follows.

\medskip We have now all elements to prove Lemma \ref{bdryflow}. Assume, by
contradiction, that for some sequence $\delta_{n}\rightarrow0$ there are
points $\xi^{n}\in\partial{\mathfrak{D}}_{\delta_{n}},$ such that $\left\vert
\varphi_{m}(\xi^{n})\right\vert \leq K$ and, for every $i\in\{1,...,m\},$
\begin{equation}
\nabla_{\xi_{i}^{n}}\varphi_{m}(\xi^{n})=0\quad\hbox{ if }\xi_{i}^{n}\in
\Omega_{\delta_{n}}, \label{Lef41}%
\end{equation}
and
\begin{equation}
\nabla_{\xi_{i}^{n}}\varphi_{m}(\xi^{n})\cdot\tau=0\quad\hbox{ if }\xi_{i}%
^{n}\in\partial\Omega_{\delta_{n}}, \label{Lef42}%
\end{equation}
for any vector $\tau$ tangent to $\partial\Omega_{\delta_{n}}$ at $\xi_{i}%
^{n}$. It follows that there exists a point $\xi_{l}^{n}\in\partial
\Omega_{\delta_{n}}$ such that $H(\xi_{l}^{n},\xi_{l}^{n})\rightarrow-\infty$
as $n\rightarrow\infty.$ Since $\left\vert \varphi_{m}(\xi^{n})\right\vert
\leq K$, there are necessarily two distint points $\xi_{i}^{n}$ and $\xi
_{j}^{n}$ coming closer to each other, that is,
\[
\rho_{n}:=\inf_{i\neq j}\abs{\xi_i^n-\xi_j^n}\rightarrow0\quad
\hbox{ as }\;n\rightarrow\infty.
\]
Without loss of generality we can assume $\rho_{n}=\abs{\xi_1^n-\xi_2^n}$. We
define $\displaystyle{x_{j}^{n}:=\frac{\xi_{j}^{n}-\xi_{1}^{n}}{\rho_{n}}}$.
Thus, up to a subsequence, there exists a $k$, $2\leq k\leq m,$ such that
\[
\lim_{n\rightarrow\infty}\abs{x_j^n}<+\infty,\;j=1,\dots,k,\quad\hbox{
and }\quad\lim_{n\rightarrow\infty}\abs{x_j^n}=+\infty,\;j>k.
\]
For $j\leq k$ we set
\[
\bar{x}_{j}=\lim_{n\rightarrow\infty}x_{j}^{n}.
\]
We consider two cases:

\noindent(1) Either
\[
\frac{\operatorname{dist}(\xi_{1}^{n},\partial\Omega_{\delta_{n}})}{\rho_{n}%
}\rightarrow+\infty,
\]
(2) or there exists $C_{0}<+\infty$ such that for almost all $n$ we have
\[
\frac{\operatorname{dist}(\xi_{1}^{n},\partial\Omega_{\delta_{n}})}{\rho_{n}%
}<C_{0}.
\]
\emph{Case 1}. It is easy to see that in this case we actually have
\[
\frac{\operatorname{dist}(\xi_{j}^{n},\partial\Omega_{\delta_{n}})}{\rho_{n}%
}\rightarrow+\infty,\quad j=1,\dots,k.
\]
Furthermore, the points $\xi_{1}^{n},\dots,\xi_{k}^{n}$ are all in the
interior of $\Omega_{\delta_{n}}$, hence (\ref{Lef41}) is satisfied for all
partial derivates $\nabla_{\xi_{j}}$, $j\leq k$. Define
\[
\tilde{\varphi}_{m}(x_{1},\dots,x_{m}):=\varphi_{m}(\xi_{1}^{n}+\rho_{n}%
x_{1},\xi_{1}^{n}+\rho_{n}x_{2},\dots,\xi_{1}^{n}+\rho_{n}x_{k},\xi_{k+1}%
^{n}+\rho_{n}x_{k+1},\dots, \xi_{m}^{n}+\rho_{n}x_{m}),
\]
and $x=(x_{1},\dots,x_{m})$. We have that, for all $l=1,2,\;j=1,\dots,k,$%
\[
\partial_{(x_{j})_{l}}\tilde{\varphi}_{m}(x)=\rho_{n}\partial_{(\xi_{j})_{l}%
}\varphi_{m}(\xi_{1}^{n}+\rho_{n}x_{1},\dots,\xi_{1}^{n}+\rho_{n}x_{k},
\xi_{k+1}^{n}+\rho_{n}x_{k+1} , \dots, \xi_{m}^{n}+\rho_{n}x_{m}).
\]
Then at $\bar{x}=(\bar{x}_{1},\dots,\bar{x}_{k},0,\dots,0)$ we have
\[
\partial_{(x_{j})_{l}}\tilde{\varphi}_{m}(\bar{x})=0.
\]
On the other hand, using Claim~\ref{Clf3} and letting $n\rightarrow\infty$, we
obtain
\[
\lim_{n\rightarrow\infty}\rho_{n}\partial_{(\xi_{j})_{l}}\varphi_{m}(\xi
_{1}^{n}+\rho_{n}\tilde{x}_{1},\dots,\xi_{m}^{n}+\rho_{n}\tilde{x}_{m}%
)=8\sum_{i\neq j,\,i\leq k}\partial_{(x_{j})_{i}}\log
\abs{\bar{x}_i-\bar{x}_j}=0,
\]
a contradiction with Claim~\ref{Clf1}.

\medskip\noindent\emph{Case 2}. In this case we actually have
\[
\frac{\operatorname{dist}(\xi_{j}^{n},\partial\Omega_{\delta_{n}})}{\rho_{n}%
}<C_{1},\quad j=1,\dots,m,
\]
for some constant $C_{1}>0$ and for almost all $n$. If the points $\xi_{j}%
^{n}$ are all interior to $\Omega_{\delta_{n}}$, we argue as in Case 1 above
to reach a contradiction to Claim~\ref{Clf2}.

\noindent Therefore, we assume that for some $j^{\ast}$ we have $\xi_{j^{\ast
}}^{n}\in\partial\Omega_{\delta_{n}}$. Assume first that there exists a
constant $C$ such that $\delta_{n}\leq C\rho_{n}$. Consider the following sum
\[
s_{n}:=\sum_{i\neq j}G(\xi_{j}^{n},\xi_{i}^{n}).
\]
In this case it is not difficult to see that $s_{n}=O(1)$ as $n\rightarrow
+\infty$. On the other hand
\[
\sum_{j}H(\xi_{j}^{n},\xi_{j}^{n})\leq H(\xi_{j^{\ast}}^{n},\xi_{j^{\ast}}%
^{n})+C\leq8\log\abs{\xi_{j^*}^n-\bar{\xi}_{j^*}^n}+C,
\]
where $\bar{\xi}_{j^{\ast}}^{n}$ is the reflection of the point $\xi_{j^{\ast
}}^{n}$ with respect to $\partial\Omega$. Since $\abs{\xi_{j^*}^n -
\bar{\xi}_{j^*}^n}\leq2\delta_{n}$ we have that
\[
\sum_{j}H(\xi_{j}^{n},\xi_{j}^{n})\rightarrow-\infty,\quad
\hbox{ as }n\rightarrow\infty.
\]
But $\left\vert \varphi_{m}(\xi^{n})\right\vert \leq K$, a contradiction.

\noindent Finally assume that $\rho_{n}=o(\delta_{n})$. In this case after
scaling with $\rho_{n}$ around $\xi_{j^{\ast}}^{n}$, and arguing similarly as
in the Case 1 we get a contradiction with Claim~\ref{Clf1} since those points
$\xi_{j}^{n}$ which lie on $\partial\Omega_{\delta_{n}}$, after passing to the
limit, give rise to points that lie on the same straight line. Thus this case
cannot occur. \qed\noindent

\bigskip

We shall now show that we can perturbe the gradient vector field of
$\varphi_{m}$ near $\partial${$\mathfrak{D}$}$_{\delta}$ to obtain a new
vector field with the same stationary points, such that $\varphi_{m}$ is a
Lyapunov function for the associated flow and {$\mathfrak{D}$}$_{\delta}%
\cap\varphi_{m}^{-1}[-K,K]$ is positively invariant.

We consider the following more general situation. Let {$U$} be a bounded open
subset of $\mathbb{R}^{N}$ with smooth boundary, and let $m\in\mathbb{N}$. We
consider a decomposition of $\overline{{U}}^{m}$ as follows. Let $S$ be the
set of all functions $\sigma:\{1,...,m\}\rightarrow\{${$U$}$,\partial${$U$%
}$\},$ and define%
\[
\mathcal{Y}_{\sigma}:=\sigma(1)\times\cdot\cdot\cdot\times\sigma
(m)\subset\mathbb{R}^{mN}.
\]
Then
\[
\overline{{U}}^{m}=%
{\textstyle\bigcup\limits_{\sigma\in S}}
\mathcal{Y}_{\sigma},\text{ \ \ \ \ }\partial({U}^{m})=%
{\textstyle\bigcup\limits_{\sigma\in S\smallsetminus\sigma_{{U}}}}
\mathcal{Y}_{\sigma},\text{ \ \ \ \ and \ \ \ \ }\mathcal{Y}_{\sigma}%
\cap\mathcal{Y}_{\zeta}=\emptyset\text{ \ if }\sigma\neq\zeta,
\]
where $\sigma_{{U}}$ stands for the constant function $\sigma_{{U}}(i)=${$U.$}
Note that $\mathcal{Y}_{\sigma}$ is a manifold of dimension $\leq mN.$ We
denote by $T_{\xi}(\mathcal{Y}_{\sigma})$ the tangent space to $\mathcal{Y}%
_{\sigma}$ at the point $\xi\in\mathcal{Y}_{\sigma}.$ The following holds.

\begin{lema}
\label{pertgrad}Let $\mathcal{F}$ be a function of class $\mathcal{C}^{1}$ in
a neighborhood of $\overline{{U}}^{m}\cap\mathcal{F}^{-1}[b,c].$ Assume that%
\begin{equation}
\nabla_{\sigma}\mathcal{F}(\xi)\neq0\text{ \ \ \ for every }\xi\in
\mathcal{Y}_{\sigma}\cap\mathcal{F}^{-1}[b,c]\text{ \ with \ }\sigma\neq
\sigma_{{U}}, \label{tangrad}%
\end{equation}
where $\nabla_{\sigma}\mathcal{F}(\xi)$ is the projection of $\nabla
\mathcal{F}(\xi)$ onto the tangent space $T_{\xi}(\mathcal{Y}_{\sigma}).$ Then
there exists a locally Lipschitz continuous vector field $\chi:\mathcal{U}%
\rightarrow\mathbb{R}^{N},$ defined in an open neighborhood $\mathcal{U}$ of
$\overline{{U}}^{m}\cap\mathcal{F}^{-1}[b,c],$ with the following properties:
For $\xi\in\mathcal{U}$,\newline(i) \ \ $\chi(\xi)=0$ if and only if
$\nabla\mathcal{F}(\xi)=0,$\newline(ii) \ $\chi(\xi)\cdot\nabla\mathcal{F}%
(\xi)>0$ if $\nabla\mathcal{F}(\xi)\neq0,$\newline(iii) $\chi(\xi)\in T_{\xi
}(\mathcal{Y}_{\sigma})$ if $\xi\in\mathcal{Y}_{\sigma}\cap\mathcal{F}%
^{-1}[b,c].$
\end{lema}

\begin{proof}
Let $\mathcal{N}_{\alpha}:=\{x\in\mathbb{R}^{N}:{\mbox {dist}}(x,\partial$%
{$U$}$)<\alpha\}.$ Fix $\alpha>0$ small enough so that there exists a smooth
retraction $r:\mathcal{N}_{\alpha}\rightarrow\partial${$U$}$.$ For every
$\sigma\in S,$ let $\widehat{\sigma}:\{1,...,m\}\rightarrow\{${$U$}%
$,\partial\mathcal{N}_{\alpha}\}$ be the function $\widehat{\sigma}%
(i)=\sigma(i)$ if $\sigma(i)=U$ and $\widehat{\sigma}(i)=\mathcal{N}_{\alpha}$
if $\sigma(i)=\partial${$U$}$.$ Set%
\[
\mathcal{U}_{\sigma}:=\widehat{\sigma}(1)\times\cdot\cdot\cdot\times
\widehat{\sigma}(m).
\]
Then $\mathcal{U}_{\sigma}$ is an open neighborhood of $\mathcal{Y}_{\sigma}.$
Let $r_{\sigma}:\mathcal{U}_{\sigma}\rightarrow\mathcal{Y}_{\sigma}$ be the
obvious retraction$.$ Assumption (\ref{tangrad}) implies that $\mathcal{F}$
has no critical points on $\partial(${$U$}$^{m})\cap\mathcal{F}^{-1}[b,c]$
and, moreover, that
\[
\nabla_{\sigma}\mathcal{F}(\xi)\cdot\nabla\mathcal{F}(\xi)>0\text{ \ \ \ if
}\xi\in\mathcal{Y}_{\sigma}\cap\mathcal{F}^{-1}[b,c]\text{ and }%
\nabla\mathcal{F}(\xi)\neq0.
\]
So taking $\alpha$ even smaller if necessary, we may assume that $\mathcal{F}$
has no critical points in $\mathcal{U}_{\sigma}\cap\mathcal{F}^{-1}[b,c] $
if\ $\sigma\neq\sigma_{{U}},$ and that%
\[
\nabla_{\sigma}\mathcal{F}(r_{\sigma}(\xi))\cdot\nabla\mathcal{F}(\xi)>0\text{
\ \ \ if }\xi\in\mathcal{U}_{\sigma}\cap\mathcal{F}^{-1}(b-\alpha
,c+\alpha)\text{ and }\nabla\mathcal{F}(\xi)\neq0.
\]
Let $\{\pi_{\sigma}:\sigma\in S\}$ be a locally Lipschitz partition of unity
subordinated to the open cover $\{\mathcal{U}_{\sigma}:\sigma\in S\}. $ Define%
\[
\chi(\xi):=\sum_{\sigma\in S}\pi_{\sigma}(\xi)\nabla_{\sigma}\mathcal{F}%
(r_{\sigma}(\xi)),\text{ \ \ \ \ \ }\xi\in\mathcal{U}:=%
{\textstyle\bigcup\limits_{\sigma\in S}}
\mathcal{U}_{\sigma}\cap\mathcal{F}^{-1}(b-\alpha,c+\alpha).
\]
One can easily verify that $\chi$ has the desired properties.
\end{proof}

As usual, set $\mathcal{F}^{c}:=\{\xi\in$ dom$\mathcal{F}:\mathcal{F}(\xi)\leq
c\}$.

\begin{lema}
[Deformation lemma]\label{def}Let $\mathcal{F}$ be a function of class
$\mathcal{C}^{1}$ in a neighborhood of $\overline{{U}}^{m}\cap\mathcal{F}%
^{-1}[b,c].$ Assume that%
\[
\nabla_{\sigma}\mathcal{F}(\xi)\neq0\text{ \ \ \ for every }\xi\in
\mathcal{Y}_{\sigma}\cap\mathcal{F}^{-1}[b,c]\text{ \ with \ }\sigma\neq
\sigma_{{U}}.
\]
If $\mathcal{F}$ has no critical points in {$U$}$^{m}\cap\mathcal{F}%
^{-1}[b,c],$ then there exists a continuous deformation $\widetilde{\eta
}:[0,1]\times(\overline{{U}}^{m}\cap\mathcal{F}^{c})\rightarrow\overline{{U}%
}^{m}\cap\mathcal{F}^{c}$such that
\begin{align*}
\widetilde{\eta}(0,\xi)  &  =\xi\text{ \ for all }\xi\in\overline{{U}}^{m}%
\cap\mathcal{F}^{c},\\
\widetilde{\eta}(s,\xi)  &  =\xi\text{ \ for all }(s,\xi)\in\lbrack
0,1]\times(\overline{{U}}^{m}\cap\mathcal{F}^{b}),\\
\widetilde{\eta}(1,\xi)  &  \in\overline{{U}}^{m}\cap\mathcal{F}^{b}\text{
\ for all }\xi\in\overline{{U}}^{m}\cap\mathcal{F}^{c}.
\end{align*}

\end{lema}

\begin{proof}
Let $\chi:\mathcal{U}\rightarrow\mathbb{R}^{N}$ be as in Lemma \ref{pertgrad}
and consider the flow $\eta$ defined by%
\begin{equation}
\left\{
\begin{array}
[c]{l}%
\frac{\partial}{\partial t}\eta(t,\xi)=-\chi(\eta(t,\xi)),\\
\text{ \ \ }\eta(0,\xi)=\xi,
\end{array}
\right.  \label{flow}%
\end{equation}
for $\xi\in\mathcal{U}$ and $t\in\lbrack0,t^{+}(\xi))$, where $t^{+}(\xi)$ is
the maximal existence time of the trajectory $t\mapsto\eta(t,\xi)$ in
$\mathcal{U}.$ For each $\xi\in\mathcal{U},$ let
\[
t_{b}(\xi):=\inf\{t\geq0:\mathcal{F}(\eta(t,\xi))\leq b\}\in\lbrack0,\infty]
\]
be the entrance time into the sublevel set $\mathcal{F}^{b}.$ Property
\emph{(ii)} in Lemma \ref{pertgrad} implies that%
\[
\frac{d}{dt}\mathcal{F}(\eta(t,\xi))=-\nabla\mathcal{F}(\eta(t,\xi))\cdot
\chi(\eta(t,\xi))\leq0,
\]
therefore $\mathcal{F}(\eta(t,\xi))$ is nonincreasing in $t.$ This, together
with $\emph{(iii)}$\ in Lemma \ref{pertgrad} yields%
\[
\eta(t,\xi)\in\overline{{U}}^{m}\cap\mathcal{F}^{-1}[b,c]\text{ \ \ if }\xi
\in\overline{{U}}^{m}\cap\mathcal{F}^{-1}[b,c]\text{ and }t\in\lbrack
0,t_{b}(\xi)].
\]
Since $\mathcal{F}$ has no critical points in {$U$}$^{m}\cap\mathcal{F}%
^{-1}[b,c],$ we have that $t_{b}(\xi)<\infty$ for every $\xi\in\overline{{U}%
}^{m}\cap\mathcal{F}^{-1}[b,c],$ and the entrance time map $t_{b}%
:\overline{{U}}^{m}\cap\mathcal{F}^{c}\cap\mathcal{U}\rightarrow
\lbrack0,\infty)$ is continuous. It follows that the map%
\[
\widetilde{\eta}:[0,1]\times(\overline{{U}}^{m}\cap\mathcal{F}^{c}%
)\rightarrow\overline{{U}}^{m}\cap\mathcal{F}^{c}%
\]
given by%
\[
\widetilde{\eta}(s,\xi):=\left\{
\begin{array}
[c]{ll}%
\eta(st_{b}(\xi),\xi) & \text{if }\xi\in(\overline{{U}}^{m}\cap\mathcal{F}%
^{c})\cap\mathcal{U}\\
\xi & \text{if }\xi\in\overline{{U}}^{m}\cap\mathcal{F}^{b}%
\end{array}
\right.
\]
is a continuous deformation of $\overline{{U}}^{m}\cap\mathcal{F}^{c}$ into
$\overline{{U}}^{m}\cap\mathcal{F}^{b}$ which leaves $\overline{{U}}^{m}%
\cap\mathcal{F}^{b}$ fixed, as claimed.
\end{proof}

\medskip\noindent\textbf{Proof of Theorem \ref{teo1}}. Fix $\delta_{1}$ small
enough so that the inclusions
\begin{equation}
{\mathfrak{D}}_{\delta_{1}}\hookrightarrow\Omega^{m}\text{ \ \ \ \ and
\ \ \ \ }{\mathfrak{D}}_{\delta_{1}}\cap D\hookrightarrow B_{\delta_{1}%
}(D):=\{x\in\Omega^{m}\,:\,{\mbox
{dist}}(x,D)\leq\delta_{1}\} \label{d1}%
\end{equation}
are homotopy equivalences, where $D:=\{\xi\in\Omega^{m}\,:\,\xi_{i}=\xi
_{j}{\mbox { for some }}i\not =j\}$. Since $\varphi_{m}$ is bounded above on
{$\mathfrak{D}$}$_{\delta_{1}}$ and bounded below on $\Omega^{m}\smallsetminus
B_{\delta_{1}}(D),$ we may choose $b_{0},c_{0}>0$ such that
\[
{\mathfrak{D}}_{\delta_{1}}\subset\varphi_{m}^{c_{0}}\text{ \ \ \ \ and
\ \ \ \ }\varphi_{m}^{b_{0}}\subset B_{\delta_{1}}(D).
\]
Fix $K>\max\{-b_{0},c_{0}\}$ and, for this $K,$ fix $\delta\in(0,\delta_{1})$
as in Lemma \ref{bdryflow}. By property (\ref{due}), for each $\varepsilon$
small enough, there exist $b<c$ such that
\[
\varphi_{m}^{c_{0}}\subset\tilde{\mathcal{F}}_{\varepsilon}^{c}\subset
\varphi_{m}^{K}\text{, \ \ \ \ }\varphi_{m}^{-K}\subset\tilde{\mathcal{F}%
}_{\varepsilon}^{b}\subset\varphi_{m}^{b_{0}},
\]
and such that, for every $\xi=(\xi_{1},...,\xi_{m})\in\partial${$\mathfrak{D}%
$}$_{\delta}$ with $\tilde{\mathcal{F}}_{\varepsilon}(\xi)\in\lbrack b,c]$
there is an $i\in\{1,...,m\}$ with%
\[%
\begin{array}
[c]{ll}%
\nabla_{\xi_{i}}\tilde{\mathcal{F}}_{\varepsilon}(\xi)\not =0 & \text{if }%
\xi_{i}\in\Omega_{\delta}\\
\nabla_{\xi_{i}}\tilde{\mathcal{F}}_{\varepsilon}(\xi)\cdot\tau\not =0\text{
\ for some }\tau\in T_{\xi_{i}}(\partial\Omega_{\delta}) & \text{if }\xi
_{i}\in\partial\Omega_{\delta}.
\end{array}
\]
We wish to prove that $\tilde{\mathcal{F}}_{\varepsilon}$ has a critical point
in $\mathfrak{D}_{\delta}\cap\tilde{\mathcal{F}}_{\varepsilon}^{-1}[b,c].$ We
argue by contradiction: Assume that $\tilde{\mathcal{F}}_{\varepsilon}$ has no
critical points in $\mathfrak{D}_{\delta}\cap\tilde{\mathcal{F}}_{\varepsilon
}^{-1}[b,c].$ Then Lemma \ref{def} gives a continuous deformation%
\[
\widetilde{\eta}:[0,1]\times(\overline{{\mathfrak{D}}}_{\delta}\cap
\tilde{\mathcal{F}}_{\varepsilon}^{c})\rightarrow\overline{{\mathfrak{D}}%
}_{\delta}\cap\tilde{\mathcal{F}}_{\varepsilon}^{c}%
\]
of $\overline{{\mathfrak{D}}}_{\delta}\cap\tilde{\mathcal{F}}_{\varepsilon
}^{c}$ into $\overline{{\mathfrak{D}}}_{\delta}\cap\tilde{\mathcal{F}%
}_{\varepsilon}^{b}$ which keeps $\overline{{\mathfrak{D}}}_{\delta}\cap
\tilde{\mathcal{F}}_{\varepsilon}^{b}$ fixed. Our choices of $b$ and $c$ imply
that $\mathfrak{D}_{\delta_{1}}\subset\overline{{\mathfrak{D}}}_{\delta}%
\cap\tilde{\mathcal{F}}_{\varepsilon}^{c}$ and $\widetilde{\eta}$ induces a
deformation of $\mathfrak{D}_{\delta_{1}}$ into $\overline{{\mathfrak{D}}%
}_{\delta}\cap\tilde{\mathcal{F}}_{\varepsilon}^{b}\subset B_{\delta_{1}}(D),$
which keeps the diagonal $D$ fixed$.$ Consequently, the homomorphism
\[
\iota^{\ast}:H^{\ast}(\Omega^{m},B_{\delta_{1}}(D))\rightarrow H^{\ast
}(\mathfrak{D}_{\delta_{1}},\mathfrak{D}_{\delta_{1}}\cap D),
\]
induced by th inclusion map $\iota:\mathfrak{D}_{\delta_{1}}\hookrightarrow
\Omega^{m},$ factors through $H^{\ast}(B_{\delta_{1}}(D),B_{\delta_{1}%
}(D))=0.$ Hence, $\iota^{\ast}$ is the zero homomorphism. On the other hand,
our choice (\ref{d1}) of $\delta_{1}$ implies that $\iota^{\ast}$ is an
isomorphism. Therefore, $H^{\ast}(\Omega^{m},B_{\delta_{1}}(D))=H^{\ast
}(\Omega^{m},D)=0.$ But, by assumption, $H^{d}(\Omega)\neq0$ for some
$d\geq1.$ If we choose $d$ so that $H^{j}(\Omega)=0$ for $j>d,$ then Lemma
\ref{monica} asserts that $H^{md}(\Omega^{m},D)\not =0.$ This is a
contradiction. Consequently, $\tilde{\mathcal{F}}_{\varepsilon}$ must have
critical point in $\mathfrak{D}_{\delta}\cap\tilde{\mathcal{F}}_{\varepsilon
}^{-1}[b,c],$ as claimed. \qed\noindent

\subsection{Proof of Theorem \ref{teo2}}

Assume that there exist an open subset $U$ of $\Omega$ with smooth boundary,
compactly contained in $\Omega,$ and two closed subsets $B_{0}\subset B$ of
$U^{m},$ which satisfy conditions \textbf{P1)} and \textbf{P2)} stated in
Section 1. By property (\ref{due}), for $\varepsilon$ small enough,
$\tilde{\mathcal{F}}_{\varepsilon}$ satisfies those conditions too, that is,%
\[
b_{\varepsilon}:=\sup_{\xi\in B_{0}}\tilde{\mathcal{F}}_{\varepsilon}%
(\xi)<\inf_{\gamma\in\Gamma}\sup_{\xi\in B}\tilde{\mathcal{F}}_{\varepsilon
}(\gamma(\xi))=:c_{\varepsilon},
\]
where $\Gamma:=\{\gamma\in\mathcal{C}(B,\overline{U}^{m}):\gamma(\xi)=\xi$ for
every $\xi\in B_{0}\}$ and, for every $\xi=(\xi_{1},...,\xi_{m})\in\partial
U^{m}$ with $\tilde{\mathcal{F}}_{\varepsilon}(\xi)\in\lbrack c_{\varepsilon
}-\alpha,c_{\varepsilon}+\alpha]$, $\alpha\in(0,c_{\varepsilon}-b_{\varepsilon
})$ small enough, one has that%
\[%
\begin{array}
[c]{ll}%
\nabla_{\xi_{i}}\tilde{\mathcal{F}}_{\varepsilon}(\xi)\not =0 & \text{if }%
\xi_{i}\in U\\
\nabla_{\xi_{i}}\tilde{\mathcal{F}}_{\varepsilon}(\xi)\cdot\tau\not =0\text{
\ for some }\tau\in T_{\xi_{i}}(\partial U) & \text{if }\xi_{i}\in\partial U,
\end{array}
\]
for some $i\in\{1,...,m\}$. If $\tilde{\mathcal{F}}_{\varepsilon}$ has no
critical points in $U^{m}\cap\tilde{\mathcal{F}}_{\varepsilon}^{-1}%
[c_{\varepsilon}-\alpha,c_{\varepsilon}+\alpha],$ then Lemma \ref{def} gives a
continuous deformation%
\[
\widetilde{\eta}:[0,1]\times(\overline{U}^{m}\cap\tilde{\mathcal{F}%
}_{\varepsilon}^{c_{\varepsilon}+\alpha})\rightarrow\overline{U}^{m}\cap
\tilde{\mathcal{F}}_{\varepsilon}^{c_{\varepsilon}+\alpha}%
\]
of $\overline{U}^{m}\cap\tilde{\mathcal{F}}_{\varepsilon}^{c_{\varepsilon
}+\alpha}$ into $\overline{U}^{m}\cap\tilde{\mathcal{F}}_{\varepsilon
}^{c_{\varepsilon}-\alpha}$ which keeps $\overline{U}^{m}\cap\tilde
{\mathcal{F}}_{\varepsilon}^{c_{\varepsilon}-\alpha}$ fixed. Let $\gamma
\in\Gamma$ be such that $\tilde{\mathcal{F}}_{\varepsilon}(\gamma(\xi))\leq
c_{\varepsilon}+\alpha$ for every $\xi\in B.$ Since $b_{\varepsilon
}<c_{\varepsilon}-\alpha,$ the map $\widetilde{\gamma}(\xi):=\widetilde{\eta
}(1,\gamma(\xi))$ belongs to $\Gamma.$ But $\tilde{\mathcal{F}}_{\varepsilon
}(\widetilde{\gamma}(\xi))\leq c_{\varepsilon}-\alpha$ for every $\xi\in B,$
contradicting the definition of $c_{\varepsilon}.$ Therefore, $c_{\varepsilon
}$ is a critical value of $\tilde{\mathcal{F}}_{\varepsilon}.$ \qed\noindent

\medskip

\bigskip\bigskip\centerline{\bf Acknowledgement}

\medskip This work has been partly supported by grants Fondecyt 1040936,
Chile, and by grants CONACYT 43724\ and PAPIIT IN105106, Mexico.

\end{document}